\documentclass{article}
\usepackage{amssymb}
\usepackage{geometry}
\usepackage{amsmath,amscd}
\usepackage{hyperref}
\usepackage{fontenc}
\usepackage{textcomp}
\usepackage[pagewise]{lineno}

\usepackage{inputenc}
		\usepackage{amsthm}

      \newtheorem{theorem}{Theorem }[section]
      \newtheorem*{theorem1}{Theorem A}
            \newtheorem{lemma}{Lemma}[section]
       
      \newtheorem{conjecture}{Conjecture}
      
      \newtheorem*{propD}{Theorem D}

			\newtheorem*{coro1}{Corollary B}
			\newtheorem*{coro2}{Corollary C}
			\newtheorem{coro}{Corollary}[section]
      \newtheorem{prop}{Proposition}[section]
      \newtheorem{defi}[theorem]{Definition}

       \numberwithin{equation}{section}

      \newtheorem{remark}{Remark}[section]
\begin{document}

%



   \begin{abstract}
 We prove an equidistribution theorem \`a la Bader-Muchnik (\cite{BM}) for operator-valued measures associated with boundary representations in the context of discrete groups of isometries of CAT(-1) spaces thanks to an equidistribution theorem of T. Roblin (\cite{Ro}). This result can be viewed as a von Neumann's mean ergodic theorem for quasi-invariant measures. In particular, this approach gives a dynamical proof of the fact that boundary representations are irreducible. Moreover, we prove some equidistribution results for conformal densities using elementary techniques from harmonic analysis.
   \end{abstract}

   AMS subject classifications: Primary 37A25, 37A30; Secondary 43A65, 43A90.

  AMS keywords: Conformal densities, boundary representations, ergodic theorems, irreducibility, equidistribution.\\
Adrien Boyer, Technion, Haifa, Israel.\\
E-mail address: adrienboyer@technion.ac.il.
   \author{Adrien Boyer \thanks{The author is supported by ERC Grant 306706.  } }
  
     \title{Equidistribution, Ergodicity and Irreducibility in CAT(-1) spaces}




   \maketitle

\section{Introduction}

Any action of a locally compact group $G$ on a measure space $(X,\mu)$ where $\mu$ is a $G$-quasi-invariant measure gives rise to a unitary representation, after renormalization with the square root of the Radon-Nikodym derivative of the action of $G$ on $(X,\mu)$. This unitary representation is called a \emph{quasi-regular representation}, and generalizes the standard notion of quasi-regular representations given by $G  \curvearrowright G/H$ where $H$ is a closed subgroup of $G$, and $G/H$ always carries a $G$-quasi-invariant measure.

The dynamical properties of the action $G \curvearrowright ( X, \mu)$ can be reflected in a such representation.

 In the context of fundamental groups of compact negatively curved manifolds, U. Bader and R. Muchnik prove in \cite[Theorem 3]{BM} an equidistribution theorem for some operator-valued measures. This theorem can be thought as a generalization of von Neumann's mean ergodic theorem for quasi-invariant measures for fundamental groups acting on the Gromov boundary of universal covers of compact negatively curved manifolds endowed with the Patterson-Sullivan measures. These quasi-regular representations are called \textit{boundary representations}. It turns out that the irreducibility of boundary representations follows from this generalization of von Neumann's ergodic theorem. We refer to  \cite{BC},\cite{BM},\cite{F1},\cite{F2},\cite{LG} and \cite{KS} for examples of natural irreducible quasi-regular representations which are related to the following conjecture:
\begin{conjecture}
For a locally compact group $G$ and a spread-out probability measure $\mu$ on $G$, the quasi-regular representation associated to a $\mu$-boundary of $G$ is irreducible.
\end{conjecture}

 In this paper, we generalize the work of U. Bader and R. Muchnik to convex cocompact groups of isometries of a \mbox{CAT(-1)} space with a non-arithmetic spectrum and to (non-uniform) lattices of a non-compact connected semisimple Lie group of rank one. Our results are based on the fundamental work of T. Roblin in \cite{Ro}. The main tool of this paper is an equidistribution theorem of T. Roblin (see Subsection \ref{BMSroblin}) which is inspired by the ideas of G. Margulis (see \cite{M}), based on the mixing property of the geodesic flow. Following the technical ideas developed in \cite{BM} and using Roblin's equidistribution theorem, we obtain a dynamical explanation of irreducibility of boundary representations in the context of CAT(-1) spaces: it comes from the mixing property of the geodesic flow. Nevertheless this approach does not work in the context of general hyperbolic groups and we refer to \cite{F1}, \cite{F2}, \cite{KS} and more recently \cite{LG} for different approaches.
 
  Moreover, we prove two equidistribution results for densities associated with the Poisson kernel and the square root of the Poisson Kernel in CAT(-1) spaces with respect to the weak* convergence of the dual space $L^{1}$ functions on the boundary.

 \subsection*{Main Results}\label{results}

The Banach space of finite signed measures on a topological compact space $Z$ is, by the Riesz representation theorem, the dual of the space of the continuous functions $C(Z)$. The Banach space of bounded linear operators from the Banach space of continuous functions to the Banach space of bounded operators on a Hilbert space will be denoted by $\mathcal{L}\big(C(Z),\mathcal{B}(\mathcal{H})\big)$. Observe that $\mathcal{L}\big(C(Z),\mathcal{B}(\mathcal{H})\big)$ is isomorphic as a Banach space to the dual of the Banach space $C(Z)\widehat{\otimes} \mathcal{H} \widehat{\otimes} \overline{\mathcal{H}}$ where $\overline{\mathcal{H}} $ denotes the conjugate Hilbert space of the complex Hilbert space $\mathcal{H}$, and $\widehat{\otimes}$ denotes the projective tensor product. Hence $\mathcal{L}\big(C(Z),\mathcal{B}(\mathcal{H})\big)$ will be called the space of \textit{operator-valued measures}.

Let $\Gamma$ be a non-elementary discrete group of isometries of $(X,d)$ a proper CAT(-1) metric space (i.e. the balls are relatively compact). We denote by $\partial X$ its Gromov boundary, and let $\overline{X}$ be the topological space $X \cup \partial X$ endowed with its usual topology that makes $\overline{X}$ compact. Recall the critical exponent $\alpha(\Gamma)$ of $\Gamma$: $$\alpha(\Gamma):=\inf \left\{s\in \mathbb{R}^{*}_{+} | \sum_{\gamma \in \Gamma} e^{-sd(\gamma x,x)} <\infty \right\}.$$Notice that the definition of $\alpha(\Gamma)$ does not depend on $x$. We assume from   now on that $\alpha(\Gamma)<\infty$.

The limit set of $\Gamma$ denoted by $\Lambda_{\Gamma}$ is the set of all accumulation points in $\partial X$ of an orbit. Namely $\Lambda_{\Gamma}:=\overline{\Gamma x}\cap \partial X$, with the closure in $\overline{X}$. Notice that the limit set does not depend on the choice of $x\in X$. Following the notations in \cite{CM}, define the geodesic hull $GH(\Lambda_{\Gamma})$ as the union of all geodesics in $X$ with both endpoints in $\Lambda_{\Gamma}$. The convex hull of $\Lambda_{\Gamma}$ denoted by $CH(\Lambda_{\Gamma})$, is the smallest subset of $X$ containing $GH(\Lambda_{\Gamma})$ with the property that every geodesic segment between any pair of points $x,y \in CH(\Lambda_{\Gamma})$ also lies in $CH(\Lambda_{\Gamma})$. We say that $\Gamma$ is \emph{convex cocompact} if it acts cocompactly on  $CH(\Lambda_{\Gamma})$.

 The translation length of an element $\gamma \in \Gamma$ is defined as $t(\gamma):=\inf \left\lbrace d(x,\gamma x),x \in X \right\rbrace.$ The \emph{spectrum} of $\Gamma$ is defined as the subgroup of $\mathbb{R}$ generated by $t(\gamma)$ where $\gamma$ ranges over the hyperbolic isometries in $\Gamma$. We say that $\Gamma$ has an arithmetic spectrum if its spectrum is a discrete subgroup of $\mathbb{R}$. We are interested in discrete groups with \emph{a non-arithmetic spectrum} because they guarantee the mixing property of the geodesic flow (see Subsection \ref{BMS}), and this condition is verified in the following cases: for isometries group of Riemannian surfaces, hyperbolic spaces and isometries groups of a CAT(-1) space such that the limit set has a non-trivial connected component. We refer to \cite{D} and to \cite[Proposition 1.6, Chapitre 1]{Ro} for more details. 

A Riemannian symmetric space $X$ of non-compact type of rank one endowed with its natural Riemannian metric is a particular case of CAT(-1) space. The space $X$ as well as its boundary $\partial X$ can be described by the quotients $X=G/K$ and $\partial X=G/Q$ where $G$ is a non-compact connected semisimple Lie group of real rank one, $K$ a maximal compact subgroup and $Q$ a minimal parabolic subgroup of $G$. A \emph{lattice} $\Gamma$ is a discrete subgroup of $G$ such that the quotient $\Gamma \backslash G$ has finite volume with respect to  the Haar measure. In this case $\Lambda_{\Gamma}=\partial X$ and $CH(\Lambda_{\Gamma})=X$. If $\Gamma \backslash G$ is a compact, we say that $\Gamma$ is a uniform lattice and this is a particular case of convex compact groups. Otherwise we say that $\Gamma$ is a \emph{non-uniform lattice}.

The foundations of Patterson-Sullivan measures theory are in the important papers \cite{Pa}, \cite{Su}. See \cite{Bou},\cite{BMo}, and \cite{Ro} for more general results in the context of CAT(-1) spaces. These measures are also called \textit{conformal densities}. \\
We denote by $M(Z)$ the Banach space of Radon measures on a locally compact space $Z$, which is identified with the dual space of compactly supported functions denoted by $C_{c}(Z)^*$, endowed with the norm $\|\mu\|=\sup\lbrace |\int_{Z}f d\mu|, \|f\|_{\infty}\leq 1,f \in C_{c}(Z) \rbrace$ where $\|f\|_{\infty}=\sup_{ z\in Z} |f(z)|$.
Recall that $\gamma_{*}\mu$ means $\gamma_{*}\mu(B)=\mu(\gamma^{-1}B)$ where $\gamma$ is in $\Gamma$ and $B$ is a borel subset of $Z$. 

We say that $\mu$ is a $\Gamma$-invariant conformal density  of dimension $\alpha \geq 0$, if $\mu$ is a map which satisfies the following conditions:
\begin{enumerate}
	\item $\mu$ is a map from $x\in X \mapsto \mu_{x} \in M(\overline{X})$, i.e. $\mu_{x}$ is a positive finite measure (density).
	\item For all $x$ and $y$ in $X$, $\mu_{x}$ and $\mu_{y}$ are equivalent, and we have $$\frac{d\mu_{y}}{d\mu_{x}}(v)=\exp{(\alpha \beta_{v}(x,y))}$$ (conformal of dimension $\alpha$).
	\item For all $\gamma \in \Gamma$, and for all $x \in X$ we have $\gamma_{*}\mu_{x}=\mu_{\gamma x}$  (invariant),
\end{enumerate}
where $\beta_{v}(x,y)$ denotes the horoshperical distance from $x$ to $y$ relative to $v$ (see Subsection \ref{CATspaces}).

	If $X$ is a \mbox{CAT(-1)} space and if $\Gamma$ is a discrete group of isometries of $X$, then there exists a $\Gamma$-invariant conformal density of dimension $\alpha(\Gamma)$ whose support is $\Lambda_{\Gamma}$. A proof can be found in \cite{Pa} and \cite{Su} for the case of hyperbolic spaces and see \cite{BM} and \cite{Bou} for the case of \mbox{CAT(-1)} spaces.

  A conformal density $\mu$ gives rise to unitary representations $(\pi_{x})_{x\in X}$ defined for $x\in X$ as:
 
	\[
		\pi_x:\Gamma\to \mathcal{U}\big(L^2(\partial X,\mu_x) \big)
	\]
	\begin{equation}\label{boundaryrep}
	   (\pi_x(\gamma)\xi)(v)=\xi(\gamma^{-1}v)\exp\left(\frac{\alpha}{2}\beta_v(x,\gamma x)\right),
	   	\end{equation}
		
where $\xi \in L^{2}(\partial X,\mu_{x})$ and $v\in \partial X$.

These representations are unitarily equivalent: the multiplication operator
\[
	U_{xy}:\xi \in L^2(\partial X,\mu_x)\to  \big(m_{xy}\cdot \xi \big)\in L^2(\partial X,\mu_y)
\]
defined by the function						   
\[
	m_{xy}(v)=\exp\bigg(-\frac{\alpha}{2} \beta_v(x,y)\bigg),
\]
is a unitary operator which intertwines the unitary representations $\pi_x$ and $\pi_y$.

 The matrix coefficient 
\begin{equation}\label{HCHfunction}
\phi_{x}: \Gamma \rightarrow \langle \pi_{x}(\gamma)\textbf{1}_{\partial X}, \textbf{1}_{\partial X}\rangle \in \mathbb{R}^{+},
\end{equation} where $\textbf{1}_{\partial X}$ is the characteristic function of $\partial X$, is called the \emph{Harish-Chandra function}.

 Pick $x$ in $X$, and a positive real number $\rho $ and define for all integers $n$ such that $n\geq \rho$ the annulus $$C_{n}(x,\rho)=\lbrace \gamma \in \Gamma | n-\rho\leq d(\gamma x, x) < n+\rho \rbrace .$$ Assume that $C_{n}(x,\rho)$ is not empty for $n\geq N_{x,\rho}$ where $N_{x,\rho}$ denotes some non-negative integer. Let $|C_{n}(x,\rho)|$ be the cardinality of $C_{n}(x,\rho)$, let $D_{y}$ be the unit Dirac mass centered at a point $y\in X$ and consider the sequence of operator-valued measures defined for all $n \geq N_{x,\rho}$ as:
\begin{equation}\label{suitesmesures} 
 \mathcal{M}^{n}_{x,\rho} :f\in C(\overline{X}) \mapsto \frac{1}{|C_{n}(x,\rho)|}\sum_{\gamma \in C_{n}(x,\rho)} D_{\gamma x} (f) \frac{\pi_{x}(\gamma)}{\phi_{x}(\gamma)} \in \mathcal{B}\big(L^{2}(\partial X,\mu_{x}) \big).
  \end{equation}
  If $f\in C(\overline{X})$, we denote by $f_{|_{\partial X}}$ its continuous restriction to the space $\partial X$.
Consider also the operator-valued measure $\mathcal{M}_{x}$ defined  as: 
\begin{equation}\label{mesurelimite}
\mathcal{M}_{x}: f\in C(\overline{X}) 
 \longmapsto \Bigg( \mathcal{M}_{x}(f):\xi \mapsto  \left( \int_{\partial X}\xi \frac{d\mu_{x}}{\|\mu_{x}\|}\right)\frac{1}{\|\mu_{x}\|}f_{|_{\partial X}} \Bigg)  \in  \mathcal{B} \big(L^{2}(\partial X,\mu_{x}) \big).
\end{equation}

\subsubsection*{The class $\mathcal{C}$:}
Let $\Gamma$ be a discrete group of isometries of a CAT(-1) space $X$. The boundary possesses a natural structure of metric space and more specifically the boundary carries a family of visual metric $(d_{x})_{x\in X}$ (see Section \ref{CATspaces}). Let $\mu=(\mu_{x})_{x\in X}$ be a conformal densities of dimension $\alpha$.\\
Let $(Z,d)$ be a compact metric measure space with a metric $d$ and a measure $\mu$. We denote by
{\rm Diam}$(Z)$ the diameter of $Z$. We say that the metric measure space $(Z,d,\mu)$ is Ahlfors $\alpha$-regular for some $\alpha>0$ if there exists a positive constant $C > 0$ such that for all $z$ in $Z$ and $0 < r < \mbox{Diam}(Z)$ we have
$$ C^{-1} r^{\alpha}\leq \mu(B(z,r))\leq Cr^{\alpha}.$$

\begin{defi}
Then we say that $\Gamma$ is in $\mathcal{C}$ if $\Gamma$ is a discrete group of isometries of a CAT(-1) space such that:
 \begin{enumerate}
 \item $\Gamma$ has \emph{non-arithmetic spectrum.}
 \item The metric measure space $(\Lambda_{\Gamma},d_{x},\mu_{x})$ is $\alpha$-Ahlfors regular for some $x$ in $X$.
 \end{enumerate}
 \end{defi}
 
 In particular the class $\mathcal{C}$ contains the convex cocompact groups of isometries of a CAT(-1) space with a non-arithmetic spectrum and the non-uniform lattices in  noncompact semisimple Lie group of rank one acting by isometries on their rank one symmetric spaces of noncompact type where $d$ is a left invariant Riemannian metric. Moreover the class $\mathcal{C}$ contains groups which are neither convex cocompact nor lattices, see Remark \ref{GMF}.\\

 The main result of this paper is the following theorem:

\begin{theorem1}\label{maintheorem} (Equidistribution \`a la Bader-Muchnik) \\Let $\Gamma$ be in $\mathcal{C}$ and let $\mu$ be a $\Gamma$-invariant conformal density of dimension $\alpha(\Gamma)$. Then for each $x$ in $ X$ there exists $\rho>0$ such that  $$\mathcal{M}^{n}_{x,\rho} \rightharpoonup \mathcal{M}_{x}$$ as $n \rightarrow + \infty$ with respect to  the weak* topology of the Banach space $\mathcal{L}\big(C(\overline{X}),\mathcal{B}(L^{2}(\partial X,\mu_{x}) )\big)$.
\end{theorem1}

With the same hypothesis of the above theorem, we deduce immediately an ergodic theorem \`a la von Neumann for the $\Gamma$-quasi-invariant measures $\mu_{x}$ on $\partial X$.\\Let $x\in X$, and denote by $\mathcal{Q}_{x}$ the orthogonal projection onto the subspace of constant functions of $L^{2}(\partial X,\mu_{x})$.

\begin{coro1}\label{maincoro}(Ergodicity \`a la von Neumann)\\
 For all $x\in X$ there exists $\rho>0$ such that $$\frac{\|\mu_{x}\|^{2}}{|C_{n}(x,\rho)|}\sum_{\gamma \in C_{n}(x,\rho)} \frac{\pi_{x}(\gamma)}{\phi_{x}(\gamma)} \rightarrow  \mathcal{Q}_{x}$$ as $n\rightarrow +\infty$ with respect to  the weak operator topology in $\mathcal{B}(L^{2}(\partial X,\mu_{x}))$. 
\end{coro1}

\begin{remark}
Consider an action of $\mathbb{Z}$ on a finite measure space $(X,\mu)$ by measure preserving transformations. Von Neumann's very well-known ergodic theorem states, in the functional analytic framework, that the ergodicity of the action is equivalent to the convergence $$\frac{1}{2n+1}\sum_{k=-n}^{n}\pi(k) \rightarrow \mathcal{Q}$$ with respect to  the weak operator topology, where $\pi$ is the quasi-regular representation obtained from the action of $\mathbb{Z}$, and where $\mathcal{Q}$ is the orthogonal projection onto the space of constant functions of the space $L^{2}(X,\mu)$. This theorem belongs to the foundation of ergodic theory and remains an important source of inspiration (see for example \cite{GN}). 

\end{remark}

With the same hypothesis of Theorem A we have:
\begin{coro2}\label{maincoro2}(Irreducibility)

For all $x\in X$, the representations $\pi_{x}:\Gamma \rightarrow \mathcal{U}(L^{2}(\partial X, \mu_{x}))$ are irreducible.
\end{coro2}
Notice that Corollary C for lattices is well known, see \cite{CS} and the method of \cite{LG} applies to the case of convex-cocompact groups. Nevertheless, this approach based on Roblin's theorem unify the irreducibility via an ergodic theorem for quasi-invariant measures and gives precise asymptotic limit of operators for the groups in the class $\mathcal{C}$. Moreover this approach seems to be the right approach to prove irreducibility in the more general context of boundary representations associated with discrete groups with parabolic elements. Note that for some discrete groups with parabolic elements acting on the hyperbolic plane $\mathbb{H}^{n}$ we obtain obtain the irreducibility of boundary representation (see Remark \ref{GMF}). The approach developped in \cite{LG} will not be fruitful whenever the group possesses parabolic elements since it works only for hyperbolic groups and this stengthens the dynamical approach of Bader and Muchnik developed here with Roblin's equisditribution theorem.
\subsubsection*{The Poisson kernel}\label{poissonker}
Recall the definition of the Poisson kernel in the context of $\mbox{CAT(-1)}$ spaces. Let $\mu$ be a $\Gamma$-invariant conformal density of dimension $\alpha$. First define
\begin{equation}\label{Poissonkernel1}
   p: (x,y,v)\in X \times X \times \partial{X} \mapsto p(x,y,v)=\exp\big( \beta_v(x,y)\big)\in \mathbb{R}^{+}.
\end{equation}
Fix $x\in X$ a base point and define the \textit{Poisson kernel associated to the measure $\mu_{x}$} as:

\begin{equation}\label{Poissonkernel}
   P: (y,v)\in X \times \partial{X} \mapsto P(y,v)=p^{\alpha}(x,y,v)=\exp\big(\alpha \beta_v(x,y)\big)\in \mathbb{R}^{+}.
\end{equation}

We follow the notations of Sj\"ogren (\cite{Sjo}) and we define for $\lambda\in\mathbb R$
and $f\in L^1(\partial X,\mu_x)$:
\begin{equation*}\label{sjogrennot}
   P_{\lambda}f(y)=\int_{\partial X}  P(y,v)^{\lambda+1/2}f(v)d\mu_x(v).
\end{equation*}

 Furthermore we denote by $\nu_{y}$ the measure associated to $P_0$ defined as 
 \begin{equation}\label{mesurenu}
 d\nu_{y}(v)=\frac{P(y,v)^{1/2}}{P_{0}\textbf{1}_{\partial X} (y)} d\mu_{x} (v).
 \end{equation}
 Observe that the measure $\nu_{y}$ is a probability measure. We refer to Subsection \ref{BMS} for the definition of the Bowen-Margulis-Sullivan measure occuring in the following statement:

\begin{propD}(Equidistribution)

$\bullet$ Let $\Gamma$ be a discrete group of isometries of a CAT(-1) space $X$ with a non-arithmetic spectrum. Let $\mu$ be a $\Gamma$-invariant conformal density of dimension $\alpha(\Gamma)$ the critical exponent of the group.
Assume that $\Gamma$ has a finite Bowen-Margulis-Sullivan measure and   assume that there exists a constant $C>0$ such that for all $x,y$ in $X$ we have $$\frac{\|\mu_{x}\|}{\|\mu_{y} \|} \leq C .$$
 Then for all $x\in X$ and for all $\rho>0$ we have $$\frac{1}{|C_{n} (x,\rho)|}\sum_{\gamma \in C_{n}(x,\rho)} \frac{\mu_{\gamma  x}}{\| \mu_{\gamma x}\|} \rightharpoonup \frac{\mu_{x}}{\| \mu_{x}\|}$$ with respect to  the weak* convergence of $L^{1}(\partial X,\mu_x)^{*}$.\\ 
 
$\bullet$ Let $\Gamma$ be in $\mathcal{C}$, then for all $x$ in $CH(\Lambda_{\Gamma})$ there exists $\rho$ such that
 
 $$\frac{1}{|C_{n} (x,\rho)|}\sum_{\gamma \in C_{n}(x,\rho)}\nu_{\gamma  x} \rightharpoonup \nu_{x}$$ with respect to  the weak* convergence of $L^{1}(\partial X,\nu_x)^{*}$.

 \end{propD}

 The method of proofs of Theorem A and D consists of two steps: given a sequence of functionals of the dual of a separable Banach space:\\

 \textbf{Step 1:} The sequence is uniformly bounded: existence of accumulation points (by the Banach-Alaoglu theorem).\\
 
 \textbf{Step 2:} Identification of the limit using equidistribution theorems (only one accumulation point).

 \subsection*{Structure of the paper} 
 
 In Section \ref{section1} we remind the reader of some standard facts about CAT(-1) spaces as well as the definition of Bowen-Margulis-Sullivan measures and Roblin's equidistribution theorem. We recall also some general facts about Banach spaces and projective tensor products, and we give a general construction of operator-valued measures that we investigate in the context of CAT(-1) spaces. \\In Section \ref{section3} we prove uniform boundedness for two sequences of functions, and we deduce \textbf{Step 1} of our results.\\ In Section \ref{section4} we use Roblin's equidistribution theorem to achieve \textbf{Step 2} of our main result. \\ In Section \ref{conclusion} we prove Theorem A, Corollary B and Corollary C. 
\\ In Section \ref{equidistribution} we prove Theorem D using the dual inequality established in Section \ref{section3}.

 \subsection*{Acknowledgements}
 I would like to thank Christophe Pittet and Uri Bader for useful discussions and criticisms. I would also like  to thank Peter Ha\"issinsky, Marc Bourdon and Davis Simmons for discussions about Ahlfors regularity. I am grateful to Felix Pogorzelski, Dustin Mayeda  and Antoine Pinochet Lobos for their remarks on this work. Finally, I would like to thank the  referee for carefully reading the paper. 
\section{Preliminaries}\label{section1}
 
 \subsection{CAT(-1) spaces}\label{CATspaces}
 In this section we survey the geometry of CAT(-1) spaces. We freely rely on \cite{Bou} where the reader could consult for further details.\\
 
A \mbox{CAT(-1)} space is a metric geodesic space such that every geodesic triangle is thinner than its comparison triangle in the hyperbolic plane, see \cite[Introduction]{BH}. Let $(X,d)$ be a proper \mbox{CAT(-1)} space.
 A geodesic ray of $(X,d)$ is an isometry: $$r:I\rightarrow X,$$ where $I=[0,+\infty)\subset \mathbb{R}$. Two geodesic rays are equivalent if the Hausdorff distance between their images are bounded, equivalently $\sup_{t\in I} d(r_{1}(t),r_{2}(t)) <+\infty$. If $r$ is a geodesic ray, $r(+\infty)$ denotes its equivalence class.  The boundary $\partial X$ is defined as the set of equivalence classes of geodesic rays. 

A geodesic segment of $(X,d)$ is an isometry: $$r:I\rightarrow X,$$ where $I=[0,a]$ with $a<\infty$. 

Fix a base point $x$. We denote by  $\mathcal{R}(x)$ the set of geodesic rays and  of geodesic segments starting at $x$ with the following convention: if $r$ is a geodesic segment defined on $[0,a]$, we set $r(t)=r(a)$ for all $t>a$. Hence we have a natural map 
\begin{align*}
\mathcal{R}(x)&\rightarrow \overline{X}=X\cup \partial X\\
r &\mapsto r(+\infty),
\end{align*} which is surjective. The set $\mathcal{R}(x)$ is endowed with the topology of uniform convergence on compact subsets of $[0,+\infty)$. By the Arzela-Ascoli theorem,  $\mathcal{R}(x)$ is a compact space. Hence, endowed with the quotient topology,  $\overline{X}$ is compact.  
Notice that the topology on $\overline{X}$ does not depend on the choice of $x$, see \cite[3.7 Proposition (1), p. 429]{BH}.\\
 
Since the CAT(-1) spaces are a particular class of general $\delta$-hyperbolic spaces we have the following inequality: for all     $x,y,z,t\in \overline{X}$ 
\begin{equation}\label{hyperbolic}
(x,z)_{t}\geq \min \lbrace (x,y)_{t},(y,z)_{t} \rbrace-\delta, 
\end{equation}
see \cite[3.17 Remarks (4), p. 433]{BH}.

Let $x$ be in $X$, and let $r$ be a geodesic ray. By the triangle inequality the function $t\mapsto  d(x,r(t))-t$ is decreasing and bounded below. Recall that the Busemann function associated to a geodesic ray $r$, is defined as the function $$b_{r}(x)=\lim_{t\rightarrow \infty} d(x,r(t))-t.$$ 

Let $x$ and $y$ be in $X$, and let $v$ be in $\partial X$. Let $r$ be a geodesic ray whose extremity is $v$, namely $r(+\infty)=v$. The limit $\lim_{t\mapsto \infty} d(x,r(t))-d(y,r(t))$ exists, is equal to $b_{r}(x)-b_{r}(y)$, and is independent of the choice of $r$. The horospherical distance from $x$ to $y$ relative to $v$ is defined as 
\begin{equation}\label{horospherical}
\beta_{v}(x,y)=\lim_{t\rightarrow \infty} d(x,r(t))-d(y,r(t)).
\end{equation}
It satisfies for all $v\in \partial X$, and for all $x,y \in X$ that
\begin{equation}\label{Bus1}
 \beta_{v}(x,y)=-\beta_{v}(y,x)
\end{equation}
\begin{equation}\label{Bus2}
\beta_{v}(x,y)+\beta_{v}(y,z)=\beta_{v}(x,z)
\end{equation}
\begin{equation}\label{Bus3}
\beta_{v}(x,y)\leq d(x,y).
\end{equation}

If $\gamma$ is an isometry of $X$ we have
\begin{equation}\label{Bus4}
\beta_{\gamma v}(\gamma x,\gamma y)=\beta_{v}(x,y).
\end{equation}
Recall that the Gromov product of two points $a,b\in X$ relative to $x\in X$
is 
\[
	(a,b)_x=\frac{1}{2}(d(x,a)+d(x,b)-d(a,b)).
\]
Let $v,w$ be in $\partial X$ such that $v\neq w$. If $a_n\to v\in\partial X$, $b_n\to w\in\partial X$, then
\[
	(v,w)_x=\lim_{n\to\infty}(a_n,b_n)_x
\]
exists and does not depend on $v$ and $w$. 

Let $r$ be a geodesic ray which represents $v$.
We have \[
	(v,y)_x=\lim_{t\rightarrow +\infty}\frac{1}{2}(d(x,r(t))+d(x,y)-d(r(t),y)),
\]
then we obtain:
\begin{equation}\label{buseman}
\beta_{v}(x,y)=2(v,y)_{x}-d(x,y).
\end{equation}

 Besides, if $q\in X$ is a point of the geodesic defined by
$v$ and $w$, then we also have:
$$
   (v,w)_x=\frac{1}{2}(\beta_v(x,q)+\beta_w(x,q)). 
$$
The formula 
\begin{equation}\label{distance}
	d_x(v,w)=\exp\big(-(v,w)_x\big)
\end{equation}
defines a distance on $\partial  X$ (we set $d_x(v,v)=0$). This is due to M. Bourdon in the context of CAT(-1) spaces, we refer to \cite[ 2.5.1 Th\'eor\`eme]{Bou} for more details. 
We have the following comparison formula:
\[
	d_y(v,w)=\exp\left(\frac{1}{2}(\beta_v(x,y)+\beta_w(x,y))\right)d_x(v,w).
\] 	

 We say that $(d_{x})_{x\in X}$ is a family of visual metrics.
 A ball of radius $r$ centered at $v\in \partial X$ with respect to  $d_{x}$ is denoted by $B(v,r)$.  A ball of radius $r$ centered at $y\in X$ is denoted by $B_{X}(y,r)$.

If $\gamma$ is an isometry of $(X, d)$, its conformal factor at $v\in\partial X$
is:
\[
	\lim_{w\to v}\frac{d_x(\gamma v,\gamma w)}{d_x(v,w)}=\exp\big(\beta_v(x,\gamma^{-1}x)\big),
\]
(see \cite[2.6.3 Corollaire]{Bou}).

If $x$ and $y$
are points of $X$ and $R$ is a positive real number, we define the shadow $$\mathcal{O}_{R}(x, y)$$
to be the set of $v$ in $\partial X$ such that the geodesic ray issued from $x$ with limit
point $v$ hits the closed ball of center $y$ with radius $R>0$.

The Sullivan shadow lemma is a very useful tool in ergodic theory of discrete groups acting on a CAT(-1) space.  See for example \cite[Lemma 1.3]{Ro} for a proof.
\begin{lemma}\label{shadowlem}(D. Sullivan)
Let $\Gamma$ be a discrete group of isometries of X. Let $\mu=(\mu_{x})_{x\in X}$ a
$\Gamma$-invariant conformal density of dimension $\alpha$. Let $x$ be in $X$. Then for $R$ large enough there exists $C>0$ such that for all $\gamma \in \Gamma$: $$\frac{1}{C}\exp\bigg({-\alpha d(x,\gamma x)}\bigg)\leq\mu_{x}\bigg(\mathcal{O}_{R}(x,\gamma  x)\bigg)\leq C \exp\bigg({-\alpha  d(x, \gamma x)}\bigg).$$
\end{lemma}

\subsection{Bowen-Margulis-Sullivan measures and Roblin's equidistribution theorem}\label{BMS}
We follow \cite[Chapitre 1 Pr\'eliminaires, 1C. Flot g\'eod\'esique]{Ro} where the reader could find more details.\\

In \cite{Su}, D. Sullivan constructs measures on the unit tangent bundle of $X$ where $X$ is the $n$-dimensional real hyperbolic space, and proves striking results for this new class of measures. We refer to \cite{Su} for more details about these measures. We recall the definitions of these analogous measures in \mbox{CAT(-1)} spaces.\\
Let $SX$ be the set of isometries from $\mathbb{R}$ to $(X,d)$ endowed with the topology of uniform convergence on compact subsets of $\mathbb{R}$. In other words, $SX$ is the set of geodesics of $X$ parametrized by $\mathbb{R}$.   We have a canonical ``projection" from $SX$ to $X$, playing the role of the projection from the unit tangent of bundle of a manifold to the manifold, which associates to $u\in SX$ a point in $X$. Indeed, notice this map in the setting of CAT(-1) spaces may be not surjective since all the geodesics are not bi-infinite.\\
 The trivial flow on $\mathbb{R}$ induces a continuous flow $(g_{t})_{t\in \mathbb{R}}$ on $SX$, called the geodesic flow. For $u\in SX$, we will denote by $g_{+\infty}(u)$ the end of the geodesic determined by $u$ for the positive time and $g_{-\infty}(u)$ the end of the geodesic for the negative time. Let $\partial^{2}X$ be the set: $\partial X \times \partial X - \left\{ (x,x)|x\in \partial X \right\}$. We recall now the so-called \emph{Hopf parametrization} in CAT(-1) spaces and to do so we fix an origin $x\in X$. We have an identification of $SX$ with $\partial ^{2}X\times \mathbb{R}$ via $$u\mapsto (g_{-\infty}(u),g_{+\infty}(u),\beta_{g_{-\infty}(u)}(u,x)).$$  Observe that $\Gamma$ acts on $\partial ^{2}X\times \mathbb{R}$ by $\gamma \cdot(v,w,s )=(\gamma v,\gamma  w,s+\beta_{v}(x,\gamma^{-1} x) )$,
and $\mathbb{R}$ acts on $\partial ^{2}X\times \mathbb{R}$ by translation $t\cdot(v,w,s )=g_{t}((v,w,s))=(v,w,s+t )$.  Notice these actions commute on $SX$.\\  
Let $\mu$ be a $\Gamma$-invariant conformal density of dimension $\alpha$. The Bowen-Margulis-Sullivan measure which is referred to as the \emph{BMS measure} $m$ on $SX$ is defined as: $$dm(u)=\frac{d\mu_{x}(v) d\mu_{x}(w) ds}{d_{x}(v,w)^{2\alpha}}\cdot $$
The measure $m$ is invariant by the action of the geodesic flow, and observe also that $m$ is a $\Gamma$-invariant measure.
We denote by $m_{\Gamma}$ the measure on the quotient $SX/\Gamma$. More precisely if $D$ is a fundamental domain for the action of $\Gamma$ on $SX$, if $h$ is a compactly supported function in $C_{c}(SX/ \Gamma)$ and if $\tilde{h}$ denotes  the lift of $h$ in $C(SX)$ we have $\int_{D}  \tilde{h}dm=\int_{SX/ \Gamma} h dm_{\Gamma} $. Moreover the quantity $\int_{D} \tilde{h} dm$ does not depend on the choice of $D$.\\
We say that $\Gamma$ admits a BMS finite measure if $m_{\Gamma}$ is finite. We denote by $g_{\Gamma}^{t}$ the geodesic flow on $SX/\Gamma$. We say that $g_{\Gamma}^{t}$ is mixing on $SX/\Gamma$ with respect to  $m_{\Gamma}$ if for all bounded Borel subsets $A,B \subset SX/\Gamma$ we have $\lim_{t \rightarrow +\infty} m_{\Gamma}(A\cap g_{\Gamma}^{t}(B))=m_{\Gamma}(A)m_{\Gamma}(B)$. 

The assumption of non-arithmeticity of the spectrum of $\Gamma$ guarantees that the geodesic flow on $SX$ satisfies the mixing property with respect to  BMS measures. We refer to \cite[Proposition 7.7]{Bab} for a proof of this fact in the case of negatively curved manifold. We refer to \cite[Chapitre 3]{Ro} for a general proof in CAT(-1) spaces.  

In \cite[Th\'eor\`eme 4.1.1, Chapitre 4]{Ro}, T. Roblin proves the following theorem based on the mixing property of the geodesic flow on $SX\backslash \Gamma$ with respect to  BMS measures:

 \begin{theorem}\label{roblin}(T. Roblin)
Let $\Gamma$ be a discrete group of isometries of $X$ with a non-arithmetic spectrum. Assume that $\Gamma$ admits a finite BMS measure associated to a $\Gamma$-invariant conformal density $\mu$ of dimension $\alpha=\alpha(\Gamma)$. Then for all $x,y \in X$ we have: $$\alpha e^{-\alpha n}||m_{\Gamma}||\sum_{ \left\{\gamma \in \Gamma|d( x,\gamma  y)< n  \right\}}D_{\gamma^{-1}  x} \otimes D_{\gamma  y} \rightharpoonup \mu_{x} \otimes \mu_{y} $$ as $n\rightarrow +\infty$  with respect to  the weak* convergence of $C(\overline{X} \times \overline{X})^{*}$.
\end{theorem}

\subsection{Operator-valued measures}\label{section2}

\subsubsection{The space of operator-valued measures as a dual space of a Banach space}
We remind to the reader why the Banach space $\mathcal{L}\big(C(Z),\mathcal{B}(\mathcal{H})\big)$ is naturally isomorphic to the dual of the Banach space $C(\partial X)\widehat{\otimes} \mathcal{H} \widehat{\otimes} \overline{\mathcal{H}}$ where $\widehat{\otimes}$ denotes the projective tensor product:

Let $E$ and $F$ be Banach spaces with norms $\|\cdot\|_{E}$ and $\|\cdot \|_{F}$. We consider the algebraic tensor product $E\otimes_{alg} F$. The projective norm of an element $g$ in $E\otimes_{alg} F$ is defined by $$\|g\|_{p}:=\inf \left\{\sum_{\mbox{\footnotesize finite}} ||e_{i}||_{E} ||f_{i}||_{F},~\mbox{such that } g=\sum_{\mbox{\footnotesize finite}}e_{i} \otimes f_{i} \right\}.$$
The projective tensor product is defined as the completion of the algebraic tensor product for the projective norm $\|\cdot \|_{p}$, and it is denoted by $$E\widehat{\otimes}F:=\overline{E\otimes_{alg} F}^{\|\|_{p}}.$$
Recall also that we have the Banach isomorphism 
\begin{equation}\label{iso1}
\mathcal{L}(E,F^{*}) \rightarrow (E\widehat{\otimes}F)^{*}
\end{equation} 
 given by: 
$$\mathcal{M} \longmapsto \big( e\otimes f \mapsto \mathcal{M}(e)f \big).$$
See \cite[p. 24]{R} for more details.

Let $\langle \cdot, \cdot \rangle$ be the inner product on $\mathcal{H}$ which is antilinear on the second variable. Define for $\xi\in \mathcal{H}$ the map $\xi^{*}\in \mathcal{H}^{*}$ which satisfies $\xi^{*}(\zeta)=\langle \zeta,\xi\rangle$ for $\zeta \in \mathcal{H}$. The canonical isomorphism between a conjugate Hilbert space and its dual is given by:
\begin{align*}
\xi \in\overline{\mathcal{H}}  \mapsto \xi^{*}  \in \mathcal{H}^{*}
\end{align*}
Define the map
\begin{align*}
\xi \otimes \eta^{*} \in \mathcal{H} \otimes \mathcal{H}^{*} \mapsto t_{\xi,\eta} \in \mathcal{B}(\mathcal{H})
\end{align*}
where $$\forall \zeta \in \mathcal{H},t_{\xi,\eta}(\zeta)=\eta^{*}(\zeta)\xi=\langle \zeta,\eta\rangle\xi.$$

Let $Tr$ be the usual semi-finite trace on $\mathcal{B}(\mathcal{H})$ and let $T$ be an operator in $\mathcal{B}(\mathcal{H})$. Notice that for all $\xi$ and $\eta $ in $\mathcal{H}$: 
\begin{equation}\label{tricky}
\langle T\xi,\eta\rangle =Tr(T t_{\xi,\eta}).
\end{equation}
It is well known that we have the isomorphism 
\begin{equation}\label{iso2'}
   \xi \otimes \eta \in \mathcal{H} \widehat{\otimes} \overline{\mathcal{H}} \mapsto t_{\xi,\eta} \in L^{1}(\mathcal{H}),
\end{equation}
where $ L^{1}(\mathcal{H})$ denotes the space of \emph{Trace class} operators.
 
Recall that we have also an isomorphism 
\begin{equation}\label{iso2}
T \in \mathcal{B}(\mathcal{H}) \mapsto  Tr_{T} \in L^{1}(\mathcal{H})^{*},
\end{equation}
where $Tr_{T}(S)=Tr(TS)$ for all $S\in L^{1}(\mathcal{H})$.

Recall that $T_{n}\rightarrow T$ with respect to  the weak operator topology if for all $\xi$ and $\eta$ in $\mathcal{H}$ we have $\langle T_{ n} \xi ,\eta \rangle \rightarrow \langle T \xi ,\eta \rangle$ as $n \rightarrow +\infty$.

\subsubsection*{An explicit isomorphism}
Let $Z$ be a compact space. The space  $\mathcal{L}\big(C(Z),\mathcal{B}(\mathcal{H})\big)$ is a Banach space with the norm $\| \mathcal{M}\|=\sup\{ \| \mathcal{M}(f)\|_{\mathcal{B}(\mathcal{H})} ,  \mbox {with } \| f\|_{\infty}\leq 1 \}$.  
Combining isomorphisms (\ref{iso1}), (\ref{iso2'}), and (\ref{iso2}) with the observation (\ref{tricky}) we obtain that the map $ \mathcal{M}\in\mathcal{L}\big(C(Z),\mathcal{B}(\mathcal{H})\big)\mapsto \widetilde{\mathcal{M}}\in \big(C(Z)\widehat{\otimes} \mathcal{H} \widehat{\otimes} \overline{\mathcal{H}}\big)^{*}$is a Banach isomorphism and satisfies for all 
$(f,\xi,\eta) \in (C(Z)\times \mathcal{H}\times \overline{\mathcal{H}})$: 
\begin{equation}\label{iso3}
\widetilde{\mathcal{M}}(f\otimes \xi \otimes\eta)=Tr(\mathcal{M}(f)t_{\xi,\eta})=\langle \mathcal{M}(f)\xi,\eta\rangle.
\end{equation}
\subsubsection{General construction of Operator-valued measures}
We give in this section a general construction of ``ergodic" operator-valued measures that we are interested in.

\subsubsection*{Quasi-regular representations}
Let $(Y,\mu)$ be a measure space.
Consider an action $\Gamma\curvearrowright (Y,\mu) $ such that $\mu$ is a finite $\Gamma$-quasi-invariant measure (i.e. $\mu$ and $\gamma_{*}\mu$ are in the same measure class). We denote by $$\frac{d\gamma_{*}\mu}{d\mu}(y)$$ the Radon-Nikodym derivative of $\gamma_{*}\mu$ with respect to  $\mu$ at a point $y$, with $\gamma$ in $\Gamma$. 
Let $\mathcal{H}$ be $L^{2}(Y,\mu)$ and for all $\xi\in \mathcal{H}$ and for all $\gamma$ in $\Gamma$ define $\pi$ to be:
$$\big(\pi(\gamma)\xi\big)(y)=\left(\frac{d\gamma_{*}\mu}{d\mu}\right)^{\frac{1}{2}}(y)\xi(\gamma^{-1} y).$$
The representation $\pi:\Gamma \rightarrow \mathcal{U}(\mathcal{H})$ is a unitary representation on the Hilbert space $\mathcal{H}$, and is called a \textit{quasi-regular} representation.
Observe that $\pi$ is a \emph{positive representation} in the sense that $\pi$ preserves the cone of positive functions.

Notice that $\pi$ extends to a representation of the group algebra denoted $\mathbb{C}\Gamma$ by
\begin{align*}
\pi:\sum c_{\gamma}\gamma  \in \mathbb{C}\Gamma \mapsto \sum{c_{\gamma}}\pi(\gamma)\in \mathcal{B}(\mathcal{H}).
\end{align*}
Define also the following matrix coefficient
\begin{align*}
\phi: \gamma \in \Gamma \mapsto \left\langle \pi(\gamma)\textbf{1}_{Y},\textbf{1}_{Y}\right\rangle \in \mathbb{R}^{+},
\end{align*}
where $\textbf{1}_{Y}$ denotes the characteristic function of the measure space $Y$. 
\subsubsection*{An ergodic operator-valued measure}\label{ergodicoperator}
Let $Z$ be a topological space and consider the space of continuous functions on $Z$ denoted by $C(Z)$. Consider a family of linear forms $\left(\ell_{\gamma}\right)_{\gamma \in \Gamma}$ on $C(Z)^{*}$.
Assume that $\Gamma$ acts isometrically on a metric space $(X,d)$. Let $x\in X$ and $\rho>0$. Define for all $n\geq \rho $ the annulus $$C_{n}(x,\rho):=\lbrace  n-\rho \leq d(\gamma x,x) <n+\rho \rbrace .$$ Assume that there exists an integer $N_{x,\rho}$ that for all $n\geq N_{x,\rho}$ the annulus $C_{n}(x,\rho)$ is not empty.
Define the sequence of operator-valued measures $(\mathcal{M}_{x,\rho} ^{n})_{n\geq N_{x,\rho}}$ as: 
\begin{align*}
\mathcal{M}_{x,\rho} ^{n}: f \in C(Z) \mapsto \frac{1}{|C_{n}(x,\rho)|}\sum_{\gamma \in C_{n}(x,\rho)}\ell_{\gamma}(f)\frac{\pi(\gamma)}{\phi(\gamma)} 
\end{align*}
and observe $$\mathcal{M}_{x,\rho} ^{n}(f) =\pi\bigg(\frac{1}{|C_{n}(x,\rho)|}\sum_{\gamma \in C_{n}(x,\rho)}\ell_{\gamma}(f)\frac{\gamma}{\phi(\gamma)} \bigg)\in  \mathcal{B}( \mathcal{H} ).$$
\subsubsection*{Properties}

Let $T$ be a bounded operator on a Hilbert space and $T^{*}$ is its adjoint.
Let $\textbf{1}_{Z}$ and $\textbf{1}_{Y}$ be the constant functions which are equal to $1$ on $Z$ and on $Y$. The Banach space $L^{\infty}(Y)$ is a Banach space with its usual norm $\| \cdot \|_{\infty}$. We denote by $\mathcal{L}(L^{\infty}(Y),L^{\infty}(Y))$ the Banach space of operators from $L^{\infty}(Y)$ to itself with the norm $ \|\cdot \|_{\mathcal{L}(L^{\infty}(Y),L^{\infty}(Y))}$.

We state some fundamental properties of the sequence $(\mathcal{M}^{n}_{x,\rho})_{n\geq N_{x,\rho}}$.

\begin{prop}\label{prop123}
Let $n$ be in a non-negative integer. Assume that $\ell_{\gamma}$ are positive linera forms (i.e. $f\geq0$ implies $\ell_{\gamma}(f)\geq0$ and for all $\gamma \in \Gamma$). We have:
\begin{enumerate}
  \item For all $f\in C(Z)$, we have $$(\mathcal{M}^{n}_{x,\rho}(f))^{*}=\left(\frac{1}{|C_{n}(x,\rho)|}\sum_{\gamma \in C_{n}(x,\rho)}\ell_{\gamma}(f)\frac{\rho(\gamma)}{\phi(\gamma)}\right)^{*}=\frac{1}{|C_{n}(x,\rho)|}\sum_{\gamma \in C_{n}(x,\rho)}\overline{\ell_{\gamma^{-1}}(f)}\frac{\rho(\gamma)}{\phi(\gamma)}.$$
	\item   $\| \mathcal{M}^{n}_{x,\rho}\|_{\mathcal{L}\big(C(Z),\mathcal{B}(\mathcal{H})\big)}\leq \|\mathcal{M}^{n}_{x,\rho}(\textbf{1}_{Z})\|_{\mathcal{B}(\mathcal{H})}$.
	\item $\|\mathcal{M}^{n}_{x,\rho}(\textbf{1}_{Z})\|_{\mathcal{L}(L^{\infty}(Y),L^{\infty}(Y))}\leq\|\mathcal{M}^{n}_{x,\rho}(\textbf{1}_{Z})\textbf{1}_{Y}\|_{\infty}$. 
\end{enumerate}
\end{prop}
The proofs are easy and left to the reader.
\section{Uniform boundedness}\label{section3}
In this section a point $x$ in $X$ is fixed.
\subsection{Useful functions}
Let $\mu$ be a $\Gamma$-invariant conformal density of dimension $\alpha$ and let $L^{\infty}(\mu)$ be the Banach space of essentially bounded functions endowed with its usual norm denoted by $\| \cdot \|_{\infty}$.
Let $\rho>0$ and assume that there exists $N_{x,\rho}$ such that $|C_{n}(x,\rho)|>0$ for all $n\geq N_{x,\rho}$.
Consider the sequence of \emph{positive functions} $F_{x,\rho}^{n}$ defined for all $n\geq N_{x,\rho}$ as: 

\begin{equation}\label{Fn}
F_{x,\rho}^{n}:v \in \partial X \mapsto \frac{1}{|C_{n}(x,\rho)|} \sum_{\gamma \in C_{n}(x,\rho)}\frac{ \exp{ \left(  \frac{\alpha}{2} \beta_{v}(x,\gamma  x)  \right) } }{\phi_{x}(\gamma)} \in \mathbb{R}^{+},
\end{equation}
where $\phi_{x}=\langle \pi_{x}(\gamma)\textbf{1}_{\partial X},\textbf{1}_{\partial X}\rangle$ is the Harish-Chandra function defined in the introduction.
Observe that $F_{x,\rho}^{n}$ is nothing but
\begin{equation}\label{normedeM}
F_{x,\rho}^{n}=\mathcal{M}_{x,\rho}^{n}(\textbf{1}_{\overline{X}})\textbf{1}_{\partial X}.
\end{equation}

Consider also the sequence of positive functions $H^{n}_{x,\rho}$ defined for all $n\geq N_{x,\rho}$ as:  

\begin{equation}\label{Hn}
H^{n}_{x,\rho}: v \in \partial X \mapsto \frac{1}{|C_{n}(x,\rho)|} \sum_{\gamma \in C_{n}(x,\rho)} \frac{ \exp{ \left(  \alpha \beta_{v}(x,\gamma  x)  \right) }}{\| \mu_{\gamma x} \|}\in \mathbb{R}^{+}. 
\end{equation}

We shall prove that $F_{x,\rho}^{n}$ and $H^{n}_{x,\rho}$ are uniformly bounded in the $L^{\infty }(\mu)$ norm. The fact that $F_{x,\rho}^{n}$ is uniformly bounded is the first step in the proof of Theorem A. 

The proof of uniform boundedness for $(F_{x,,\rho}^{n})_{n\geq N_{x,\rho}}$ consists in two parts: we shall obtain sharp estimates of Busemann functions on shadows, then use Ahlfors regularity condition to estimate the Harish-Chandra function $\phi_{x}$. 

The method of the proof of the uniform boundedness of $(F_{x,,\rho}^{n})_{n\geq N_{x,\rho}}$ applies for showing the uniform boundedness of $(H_{x,\rho}^{n})_{n\geq N_{x,\rho}}$ with suitable hypothesis.

 \subsection{Estimates for Busemann functions}
These techniques using the hyperbolic inequality (\ref{hyperbolic}) extended to the whole space $\overline{X}$ are very powerful. See for example  \cite{CM} and \cite{BM} where these techniques are used. 
\begin{lemma}\label{encadrement buseman dur}
Let $R>0$ and let $v\in \partial X$. We have for all $y\in X$ and for all $w$ in $\mathcal{O}_{R}(x,y)$: 
$$\min \lbrace (w,v)_{x},d(x,y)\rbrace-R-\delta  \leq  (v,y)_{x}\leq (v,w)_{x}+R+\delta.$$
\end{lemma}
\begin{proof}
Recall that $d(x,y)-2R\leq \beta_{w}(x,y)\leq d(x,y)$ for all $w\in \mathcal{O}_{R}(x,y)$. Hence, by equation (\ref{buseman}), we have
\begin{equation}\label{encadre}
d(x,y)-R\leq (w,y)_{x}\leq d(x,y).
\end{equation}
On one hand,  using first the hyperbolic property (\ref{hyperbolic}), then the observation (\ref{encadre})  we have 
\begin{align*}
(v,y)_{x}&\geq \min \lbrace(v,w)_{x},(w,y)_{x}\rbrace-\delta \\
& \geq \min \lbrace (w,v)_{x},d(x,y)\rbrace-R-\delta. 
\end{align*}

On the other hand, using $(v,y)_{x}\leq d(x,y)$ we have
\begin{align*}
(v,w)_{x}&\geq \min \lbrace(v,y)_{x},(y,w)_{x}\rbrace-\delta \\
& \geq \min \lbrace(v,y)_{x},d(x,y)-R\rbrace-\delta \\
& \geq (v,y)_{x}-R-\delta. 
\end{align*}
\end{proof}

\begin{prop}\label{maj buseman}
Let $R>0$ and let $n$ be a non-negative integer such that $n\geq \rho$ and let $v \in \partial X$.  There exists $q_{v}$ in $X$ satisfying $d(x,q_{v})=n+\rho$, such that for all $y$ in $X$ with $n-\rho\leq d(x,y)< n+\rho $ and for all $w$ in $\mathcal{O}_{R}(x,y)$ we have $$\beta_{v}(x,y)\leq \beta_{w}(x,q_{v})+2(R+\rho)+4\delta.$$
\end{prop}
\begin{proof}

Define $q_{v}$ as the point on the unique geodesic passing through $v$ and $x$ such that $d(x,q_{v})=n+\rho$.

 Since $(v, y)_{x} \leq d(x,y) $, the right hand side inequality of Lemma \ref{encadrement buseman dur}, the definition of $q_{v}$ combined with the hyperbolic inequality (\ref{hyperbolic}) imply for all $w$ in $\mathcal{O}_{R}(x,y)$  that
\begin{align*}
(v,y)_{x}&\leq \min \lbrace(v,w)_{x},d(x,y)\rbrace+R+\delta \\
& \leq \min \lbrace (v,w)_{x},d(x,q_{v})\rbrace+R+\delta \\
&= \min \lbrace (v,w)_{x},(v,q_{v})_{x}\rbrace+R+\delta\\
& \leq (w,q_{v})_{x}+R+2\delta.
\end{align*}

Since $y$ satisfies $n-\rho\leq d(x,y)< n+\rho $ and $d(x,q_{v})=n+\rho$ the above inequality implies
\begin{align*}
\beta_{v}(x,y)\leq \beta_{w}(x,q_{v})+2(\rho+R)+4\delta.
\end{align*} 
\end{proof}

\subsection{  Ahlfors regularity and Harish-Chandra functions}\label{HarishA}

Let $\Gamma$ be a discrete group of isometries of a CAT(-1) space $X$ and let $\mu$ be a $\Gamma$-invariant conformal density of dimension $\alpha$.
Fix a point $x$ in $X$ and define the function  
\begin{equation}\label{HCHcontinue}
\varphi_{x}:y\in X \mapsto \int_{\partial X}\exp \bigg( \frac{\alpha}{2} \beta_{v}(x,y)\bigg)d\mu_{x}(v).
\end{equation}
Observe that $\phi_{x}$ is the restriction of $\varphi_{x}$ to the orbit $\Gamma x$.\\
Let $\mathcal{Y}$ be a subset of $X$. We say that $\varphi_{x}$ satisfies the \emph{Harish-Chandra estimates on $\mathcal{Y}$} if there exist two polynomials $Q_{1}$ and $Q_{2}$ of \emph{degree one} such that for all $y\in \mathcal{Y}$ we have $Q_{1}(d(x,y))>0$ and 
\begin{equation}\label{HCHestim}
Q_{1}\big( d(x,y)\big) \exp{\bigg(- \frac{\alpha}{2}d(x,y) \bigg)}\leq \varphi_{x}(y) \leq Q_{2}\big( d(x,y)\big) \exp{\bigg( -\frac{\alpha}{2}d(x,y) \bigg)}.
\end{equation}
\\
 Let $R>0$ and such that for all $x$ and $y$ in $X$ the shadows $\mathcal{O}_{R}(x,y)$ are not empty. Pick a point $w_{x}^{y}$  in $\mathcal{O}_{R}(x,y)$. In the context of negatively curved manifolds, we can think about $w_{x}^{y}$ as the ending point of the geodesic passing through $x$ and $y$, oriented from $x$ to $y$. 

\begin{lemma}\label{ineghch}
Let $v \in \partial X$ and $y\in X$. Let $w_{x}^{y}$ be a point in $\mathcal{O}_{R}(x,y)$. Then, we have $$  \exp{\left(  \frac{\alpha}{2}\beta_{v}(x,y) \right) }\leq \exp{\big(\alpha (\delta+R)\big)}\exp{\bigg(-\frac{\alpha}{2} d(x,y)\bigg)}\frac{1}{d_{x}^{\alpha}\big(v,w_{x}^{y}\big)}, $$ and 
$$\exp{\left(  \frac{\alpha}{2}\beta_{v}(x,y) \right) } \geq \exp{\big(-\alpha (\delta+R)\big)} \exp\bigg(-\frac{\alpha}{2} d(x,y)\bigg)\bigg( \min \bigg\{ \frac{1}{d_{x}(v,w_{x}^{y})^{\alpha}},\exp{\big(\alpha d(x,y)\big)} \bigg \}\bigg).$$
\end{lemma}

\begin{proof}
We prove the first inequality.
 The right hand side inequality of Lemma \ref{encadrement buseman dur} leads to
\begin{align*}
(v,y)_{x}\leq (v,w_{x}^{y})_{x}+R+\delta.
\end{align*}

Combining this inequality with equation (\ref{buseman}), we have 
\begin{align*}
\exp{ \left(  \frac{\alpha}{2}\beta_{v}(x,y) \right) }&\leq \exp{ \big(\alpha (\delta+R)\big)} \exp{ \bigg( \alpha (v,w_{x}^{y})_{x}\bigg)} \exp{ \left( -\frac{\alpha}{2}d(x,y)\right)}.
\end{align*}
The definition (\ref{distance}) of the visual metric completes the proof.\\
The left hand side of the inequality of Lemma \ref{encadrement buseman dur} gives the other inequality.
\end{proof}

\begin{prop}\label{H-CHestimates}
Let $\mu$ be a $\Gamma$-invariant conformal density of dimension $\alpha$. Assume that $\big(\Lambda _{\Gamma}, d_{x},\mu_{x} \big)$ is Ahlfors $\alpha$-regular. Then there exists $R_{x}>0$ such that the function $\varphi_{x}$ satisfies the Harish-Chandra estimates on $\Gamma x \backslash B_{X}(x,R_{x})$.

Moreover, if $\Gamma$ is convex cocompact there exists $R_{x}>0$ such that the function $\varphi_{x}$ satisfies the Harish-Chandra estimates on $CH(\Lambda_{\Gamma})\backslash B_{X}(x,R_{x})$.
\end{prop}

\begin{proof}

We first prove the right hand side inequality of (\ref{HCHestim}) on $\mathcal{Y}=\Gamma x $.
Let $\gamma \in \Gamma$, and consider a point $w_{x}^{\gamma x}\in \mathcal{O}_{R}(x,\gamma x) \cap \Lambda _{\Gamma}$.
Consider the ball of $\partial X$ of radius   $\exp{\big(-d(x,\gamma x)\big)}$ with respect to   $d_{x}$ centered at $w_{x}^{\gamma x}$ denoted by $$B_{\gamma}:=B\bigg(w_{x}^{\gamma x},\exp{\big(-d(x,\gamma x)\big)}\bigg).$$ 

\begin{align*}
\phi_{x}(\gamma)&=\int_{\partial X}\exp{\left( \frac{\alpha}{2} \beta_{v}(x,\gamma x)\right) }d\mu_{x}(v)\\ 
&=\int_{B_{\gamma}}\exp{\left( \frac{\alpha}{2} \beta_{v}(x,\gamma x)\right) }d\mu_{x}(v)+\int_{\partial X \backslash B_{\gamma}}\exp{\left( \frac{\alpha}{2} \beta_{v}(x,\gamma x)\right) }d\mu_{x}(v).
\end{align*}

 Ahlfors $\alpha$-regularity implies for the first term that there exists $C>0$ such that 
  \begin{align*}
 \int_{B_{\gamma}}\exp{\left( \frac{\alpha}{2} \beta_{v}(x,\gamma x)\right) }d\mu_{x}(v)&\leq \mu_{x}\big(B_{\gamma}\big)\exp{\left( \frac{\alpha}{2} d(x,\gamma x)\right) } \\
 \leq C\exp{\left( -\frac{\alpha}{2} d(x,\gamma x)\right) }.
 \end{align*}
 The right hand side inequality of Lemma \ref{ineghch} implies that
\begin{align*}
\int_{\partial X \backslash B_{\gamma}}\exp{\left( \frac{\alpha}{2} \beta_{v}(x,\gamma x)\right) }d\mu_{x}(v)&\leq C_{\alpha,\delta,R}\exp{ \left( -\frac{\alpha}{2}d(x,\gamma x)\right)}\int_{ \partial X \backslash B_{\gamma}}\frac{1}{d_{x}^{\alpha}\big(v, w_{x}^{\gamma x}\big)}d\mu_{x}(v),
\end{align*}
for some positive constant $C_{\alpha,\delta,R}>0$.

Write now
\begin{align*}
\int_{ \partial X \backslash B_{\gamma}}\frac{1}{d_{x}^{\alpha}\big(v,w_{x}^{\gamma x}\big)}d\mu_{x}(v)&=\int_{\mathbb R} \mu_{x}\bigg(\bigg\{ v \in \partial X\big |  \frac{1}{d_{x}^{\alpha}\big(v,w_{x}^{\gamma x})}>t  \bigg \} \bigg)dt \\
&=\int_{1/D^{\alpha}}^{\exp{\alpha d(x,\gamma x)}} \mu_{x}\bigg(\bigg\{ v \in \partial X\big |  d_{x}(v,w_{x}^{\gamma x})<\frac{1}{t^{1/ \alpha}}  \bigg \} \bigg)dt \\
&\leq \sum_{n=p}^{N} \mu_{x}\bigg(\bigg\{ v \in \Lambda_{\Gamma} \big |  d_{x}\big(v,w_{x}^{\gamma x})<\frac{1}{n^{1/ \alpha}}  \bigg \} \bigg)
\end{align*}
where $D$ denotes $\mbox{Diam}(\partial X)$, $p$ the integer part of $\frac{1}{D^{\alpha}}$ and $N$ then integer part of $\exp{(\alpha d(x,\gamma x))}+1$.
Ahlfors regularity implies that there exists $C>0$ such that $$\mu_{x}\bigg(\bigg\{ v \in \Lambda_{\Gamma} \big |  d_{x}\big(v,w_{x}^{\gamma x})<\frac{1}{n^{1/ \alpha}}  \bigg \} \bigg)\leq  \frac{C}{n} .$$ Hence there exists a constant $C'>0$ such that $$\int_{ \partial X \backslash B_{\gamma}}\frac{1}{d_{x}^{\alpha}\big(v,w_{x}^{\gamma x}\big)} d\mu_{x}(v)\leq \alpha d(x,\gamma x)+C'.$$
Hence, we have 
$$
\int_{\partial X \backslash B_{\gamma}}\exp{\left( \frac{\alpha}{2} \beta_{v}(x,\gamma x)\right) }d\mu_{x}(v) \leq C_{\alpha,\delta,R}  \exp{ \left( -\frac{\alpha}{2}d(x,\gamma x)\right)} (\alpha d(x,\gamma x)+C').$$

Therefore, we have found a polynomial of degree one such that $\varphi_{x}$ satisfies the (right hand side) Harish-Chandra estimates on $\Gamma x$. The left hand side of Harish-Chandra estimates on $\Gamma x$ is analogous by the second inequality of Lemma  \ref{ineghch}, but the constant term of the polynomial $Q_1$ might be non-positive. Hence the Harish-Chandra estimates hold only on $\Gamma x \backslash B_{X}(x,R_{x})$ for some positive number $R_{x}$.\\

Assume that $\Gamma$ is convex cocompact. We shall estimate $\varphi_{x}$ on $CH(\Lambda_{\Gamma})$.  Let $y\in CH(\Lambda_{\Gamma})$ and pick a fundamental domain $ D_{\Gamma}\subset CH(\Lambda_{\Gamma})$ relatively compact, and consider $D_{\Gamma}'$ a relatively compact neighborhood  of $x$ which contains $D_{\Gamma}$. Then there exists $\gamma \in \Gamma$ such that $y\in \gamma D_{\Gamma}'$. Thanks to the cocycle identity (\ref{Bus2}) we have  
 \begin{align*}
 \varphi_{x}(y)&=\int_{\partial X}\exp \bigg(  \frac{\alpha}{2} \beta_{v}(x,\gamma x)\bigg) \exp \bigg(  \frac{\alpha}{2} \beta_{v}(\gamma x, y)\bigg)d\mu_{x}(v).
  \end{align*}
Thanks to the properties of Busemann functions (\ref{Bus1}) and (\ref{Bus3}), observe that   $$ \exp \bigg( - \frac{\alpha}{2}\mbox{Diam}(D'_{\Gamma})\bigg) \phi_{x}(\gamma)\leq   \varphi_{x}(y)\leq   \phi_{x}(\gamma) \exp \bigg(  \frac{\alpha}{2}\mbox{Diam}(D'_{\Gamma})\bigg)$$

Observe also that $$ d(x,y)-\mbox{Diam}(D'_{\Gamma})\leq d(\gamma x, x)\leq d(x,y)+\mbox{Diam}(D_{\Gamma}')$$ for all $y\in X$ such that $d(x,y)\geq \mbox{Diam}(D'_{\Gamma})$. Since $\varphi_{x}$ satisfies the Harish-Chandra estimates on $\Gamma x \backslash B_{X}(x,R_{x})$ we have the Harish-Chandra estimates of $\varphi_{x}$ on $CH(\Lambda_{\Gamma})\backslash B_{X}(x, R_{x}')$ where $R_{x}'=\max{\{R_{x}, \mbox{Diam} (D'_{\Gamma})\}}$ and the proof is done. 

\end{proof}

\begin{remark}
Notice that a slight modification of the first part of this proof gives a geometrical proof of the Harish-Chandra estimates of the $\Xi$--Harish-Chandra function in the context of rank one semisimple Lie groups (see \cite{An} and \cite{GV}). It would be interesting to study an analog of Harish-Chandra estimates on $CH(\Lambda_{\Gamma})$ for Harish-Chandra functions associated with geometrically finite groups with parabolic elements. \end{remark}

\begin{remark}\label{GMF}
In \cite[Theorem 2]{SV}, B. Stratmann and S.-L. Velani prove, in the context of hyperbolic plane $\mathbb{H}^{n}$, the so-called \emph{Global Measure Formula} for 
geometrically finite groups with parabolic elements. A conformal density of a geometrically finite group with parabolic elements is Ahlfors regular if and only if all the parabolic cusps have the same rank and that rank is equal to the critical exponent of the group. Hence such geometrically finite groups belong to the class $\mathcal{C}$ since their spectrum are non-arithmetic. 
\end{remark}

\subsection{Uniform boundedness}

\begin{prop}\label{inegsample}
 Let $\mu$ be a $\Gamma$-invariant conformal density of dimension $\alpha(\Gamma)$ where 
 $\Gamma$ is in $\mathcal{C}$. Then there exists $\rho$ and an integer $N$ such that for all $n\geq N$, the sequence $F^{n}_{x,\rho}$ is uniformly bounded in the $L^{\infty}(\mu)$ norm.
  
  \end{prop}
\begin{proof} 
If two sequences $u_{n}$ and $v_{n}$ of positive real numbers satisfy  $\lim_{n \to \infty} u_{n} / v_{n}=1$ we write $u_{n}\sim v_{n}$.

We shall prove first that $C_{n}(x,\rho)$ is not empty, at least for $n$ large enough.
For a non negative integer $n$, set $\Gamma_{n}(x):= \left\{\gamma \in \Gamma|d( x,\gamma  x)<  n \right\}$.
  Applying Theorem \ref{roblin} to the function $\textbf{1} _{\overline {X}}\otimes \textbf{1}_{\overline {X}}$ we obtain as $n\rightarrow +\infty $ $$|\Gamma_{n}(x)| \sim \frac{\exp(\alpha n)\| \mu_{x}\|^{2}}{\alpha \|m_{\Gamma} \|},$$ and thus as $n\rightarrow +\infty $
 \begin{equation}\label{estimcouronne}
 |C_{n}(x,\rho)| \sim \frac{\exp(\alpha n)(2\sinh (\alpha \rho))\| \mu_{x}\|^{2}}{\alpha \|m_{\Gamma} \|}\cdot 
 \end{equation}
Hence, for all $\rho$ there exists $N_{x,\rho}$ such that for all $n\geq N_{x,\rho}$ we have $|C_{n}(x,\rho)|>0$.

There are two steps:
\\

\textit{Step 1:
Assume that $x$ is in $CH(\Lambda_{\Gamma})$. Then for all $\rho>0$, there exists $N^{'}_{x,\rho}$and for all $n\geq N^{'}_{x,\rho} $ the sequence  $F_{x,\rho}^{n}$ is uniformly bounded with respect to  the $L^{\infty}(\mu)$ norm.}
\\

First of all let $R_{x}$ be a positive real number such that the Harish-Chandra estimates hold on $CH(\Lambda_{\Gamma}) \backslash R_{x}$. Let $\rho>0$ and let $N_{x,\rho}$ be an integer such that for all $n\geq N_{x,\rho}$ we have $C_{n,\rho}(x) \subset CH(\Lambda_{\Gamma}) \backslash R_{x}$. Let $v$ be in $\Lambda_{\Gamma}$, then Proposition \ref{maj buseman} provides $q_{v}\in X$ with $d(x,q_{v})=n+\rho$, such that for all $\gamma  \in C_{n,\rho}(x)$:

\begin{equation}
\exp{ \left(  \frac{\alpha}{2} \beta_{v}(x,\gamma  x) \right)\leq \frac{\exp\big(\alpha (2(R+\rho)+4\delta)\big)}{\mu_{x}(\mathcal{O}_{R}(x,\gamma x))} \int_{\mathcal{O}_{R}(x,\gamma x)}} \exp{ \left(  \frac{\alpha}{2} \beta_{w}(x,q_{v})  \right) }d\mu_{x}(w).
\end{equation}
We set for the following computation $C_{0}:=\exp\big(\alpha (2(R+\rho)+4\delta)\big)$.\\
 Therefore we have for $v\in \Lambda_{\Gamma}:$
\begin{align*}
F_{x,\rho}^{n}(v)&=\frac{1}{|C_{n}(x,\rho)|} \sum_{\gamma \in C_{n}(x,\rho)}\frac{ \exp{ \left(  \frac{\alpha}{2} \beta_{v}(x,\gamma  x)  \right) } }{\phi_{x}(\gamma)} \\
&\leq \frac{C_{0}}{|C_{n}(x,\rho)|} \sum_{\gamma \in C_{n}(x,\rho)} \frac{  \int_{\mathcal{O}_{R}(x,\gamma  x)} \exp{ \left(  \frac{\alpha}{2} \beta_{w}(x,q_{v})  \right)  }d\mu_{x}(w)}{\mu_{x}\big(\mathcal{O}_{R}(x,\gamma  x)\big) \phi_{x}(\gamma)}\\
&\leq C\frac{C_{0}}{\exp{(-\alpha (n-\rho))} |C_{n}(x,\rho)|}\left( \sup_{\gamma \in C_{n}(x,\rho)}{\frac{1}{\phi_{x}(\gamma)}} \right) \sum_{\gamma \in C_{n}(x,\rho)}  \int_{\mathcal{O}_{R}(x,\gamma  x)} \exp{ \left(  \frac{\alpha}{2} \beta_{w}(x,q_{v})  \right) } d\mu_{x}(w)\\
&\leq C\frac{C_{0}}{ \exp{(-\alpha (n-\rho))} |C_{n}(x,\rho)|}\left( \sup_{\gamma \in C_{n}(x,\rho)}{\frac{1}{\phi_{x}(\gamma)}} \right)  \big(m\times  \varphi_{x}(q_{v})\big),
\end{align*}
 where the last inequality follows from the fact that there exists an integer $m$ such that for all $w\in \partial X $ the cardinality of $\lbrace \gamma \in C_{n}(x,\rho)| w \in \mathcal{O}_{R}(x,\gamma x)\rbrace$ is bounded by $m$.
 
The Sullivan Shadow lemma (for $R$ large enough) implies that there exists $c'>0$ such that for all $n$ big enough we have :  
 \begin{equation*}
\exp{(-\alpha (n-\rho))} |C_{n}(x,\rho)| \geq c'.
 \end{equation*}

Since the hypothesis guarantee the Ahlfors regularity of the limit set for the groups in the class $\mathcal{C}$ ( \cite[2.7.5 Th\'eor\`eme]{Bou} for convex cocompact groups, the case of lattices is well known) then Proposition \ref{H-CHestimates} implies that there exists $C'>0$, such that for $q_{v}\in  CH(\Lambda_{\Gamma}) \backslash R_{x}$ with $d(x,q_{v})=n+\rho$ we have $$\left( \sup_{\gamma \in C_{n}(x,\rho)}{\frac{1}{\phi_{x}(\gamma)}} \right) \varphi_{x}(q_{v})\leq C'.$$

 Hence for  $x\in CH(\Lambda_{\Gamma})$ and for all $\rho>0$, there exists $K>0$ and $N^{'}_{x,\rho}$ such that for all $n\geq N^{'}_{x,\rho}$ we have $$\|F^{n}_{x,\rho}\|_{\infty}\leq K.$$

 \textit{Step 2: Assume that $x$ is in $X\backslash CH(\Lambda_{\Gamma})$. There exist $\rho'_{x}>0$ and an integer $N^{'}_{x,\rho'_{x}}$ such that for all $n\geq N^{'}_{x,\rho'_{x}} $ the sequence  $F_{x,\rho}^{n}$ is uniformly bounded with respect to  the $L^{\infty}(\mu)$ norm.}\\

 Fix $\rho>0$ and let $x_{0}$ be the projection of $x$ in $CH(\Lambda_{\Gamma})$ and set  $$\kappa:= d\big(x,CH(\Lambda_{\Gamma})\big)=d(x,x_{0}).$$
Using the relations (\ref{Bus2}), (\ref{Bus4}), (\ref{Bus1}), and (\ref{Bus3}) we obtain $$\phi_{x}(\gamma)\geq \exp{(\alpha \kappa )}\phi_{x_{0}}(\gamma) .$$
 Observe that $ C_{n}(x,\rho)\subset C_{n}(x_{0},\rho+2\kappa)$. We have:
 \begin{align*}
 F^{n}_{x,\rho}(v)&=\frac{1}{|C_{n}(x,\rho)|} \sum_{\gamma \in C_{n}(x,\rho)}\frac{ \exp{ \left(  \frac{\alpha}{2} \beta_{v}(x,\gamma  x)  \right) } }{\phi_{x}(\gamma)}\\
 &=\frac{1}{|C_{n}(x,\rho)|} \sum_{\gamma \in C_{n}(x,\rho)}\frac{ \exp{ \left(  \frac{\alpha}{2} \beta_{v}(x,  x_{0})  \right) } \exp{ \left(  \frac{\alpha}{2} \beta_{v}(x_{0},\gamma  x_{0})  \right) } \exp{ \left(  \frac{\alpha}{2} \beta_{v}(\gamma x_{0},\gamma  x)  \right) } }{\phi_{x}(\gamma)}\\
 &\leq   \frac{\exp{ (2\alpha \kappa)}}{|C_{n}(x,\rho)|} \sum_{\gamma \in C_{n}(x_{0},\rho+2\kappa)}\frac{ \exp{ \left(  \frac{\alpha}{2} \beta_{v}(x_{0},\gamma  x_{0})  \right) }  }{\phi_{x_{0}}(\gamma)}\\
 &= \bigg( \exp{ (2\alpha \kappa)}  \frac{|C_{n}(x_{0},\rho+2\kappa)|}{|C_{n}(x,\rho)|} \bigg) F_{x_{0},\rho+2\kappa}^{n},
 \end{align*}
 where the third inequality comes from the relations (\ref{Bus4}) and (\ref{Bus3}).
 Since $|C_{n}(x_{0},\rho+2\kappa)| / |C_{n}(x,\rho)| $ is bounded above by some constant depending on $\rho$ and $\kappa$, we apply \textit{Step 1} to $F_{x_{0},\rho +2\kappa}^{n}$ with $x_{0}$ and $\rho+2\kappa$ to complete the proof.

 \end{proof}
 \begin{remark}
If for any choice of an origin $x$ the metric measure space $(\Lambda_{\Gamma},d_{x},\mu_{x})$  is Ahlfors regular and if $CH(\Lambda_{\Gamma})=X$, then the above proposition holds for all $\rho>0$ independently of the choice of $x$. These conditions include the case of lattices in rank one semisimple Lie groups and fundamental groups of compact negatively curved manifolds. 
\end{remark}

\begin{remark}
For a proof of this uniform boundedness in the context of hyperbolic groups we refer to \cite[Proposition 5.2]{LG}. 
\end{remark}

\begin{prop}\label{inegsample2}
 Let $\mu$ be $\Gamma$-invariant conformal density of dimension $\alpha(\Gamma)$ the critical exponent of the group and let $\Gamma$ be a discrete group of isometries of a CAT(-1) space $X$ with a non-arithmetic spectrum with a finite BMS measure. Assume that there exists $C>0$ such that for all $y\in X$ we have $\|\mu_{y}\| / \|\mu_{x}\| \leq C$. For all  $\rho>0$, there exists an integer $N$ such that for all $n\geq N$ the sequence of functions $H^{n}_{x,\rho}$ is uniformly bounded in the $L^{\infty}(\mu)$ norm.
\end{prop}
The proof for $H^{n}_{x,\rho}$ follows the same method and is left to the reader. Notice that this proof is easier because it does not deal with the Harish-Chandra estimates.

\section{Analysis of matrix coefficients}\label{section4}
In this section we fix $x$ as an origin of $X$.
\subsection{Notation} Let $\Gamma$ be a discrete group of isometries of $X$ and let $\mu$ be a $\Gamma$-invariant conformal density of dimension $\alpha.$ Let $A$ be a subset of $\partial X$ and $a>0$ positive real number and define $A_{x}(a)$ the subset of $\partial X$ as $$ A_{x}(a)=\lbrace  v | \inf_{w\in A_{}}d_{x}(v,w) <\exp({-a})\rbrace .$$
We will write $A(a)$ instead of $A_{x}(a)$. Recall that $\cap_{a>0}A(a)=\overline{A}$.
\\

Let $R$ a positive real number and define the cone of base $A$ to be $$C_{R }(x,A):=\lbrace y\in X|\exists v \in A \mbox{ satisfying } [xv)\cap B(y,R)\neq \varnothing \rbrace,$$ where $[xv)$ represents the unique geodesic passing through $x$ with the ending point $v\in \partial X$. In other words we have:
\begin{equation}
C_{R }(x,A):=\lbrace y\in X| \mathcal{O}_{R}(x,y)\cap A \neq \varnothing \rbrace.
\end{equation}

Define $b_{x}(y)$ the function 
\begin{align}
b_{x}(y):v\in \partial X \mapsto \exp\bigg(\frac{\alpha}{2}\beta_{v}(x,y) \bigg).
\end{align}
Notice that $\varphi_{x}(y)=\int_{\partial X} b_{x}(y) (v)d\mu_{x}(v)$.
\subsection{Sharp estimates}
Assume that $\varphi_{x}$ satisfies Harish-Chandra estimates on $\mathcal{Y}$.

\begin{lemma}\label{majorationendehors}

Let $A$ be a Borel subset of $\partial X$ and let $a>0$. There exists a constant $C_{0}$ such that for all $y$ in $\mathcal{Y}$ satisfying $\mathcal{O}_{R}(x,y) \cap A(a)=\varnothing$, we have $$\frac{\langle b_{x}(y),\chi_{A}\rangle}{\varphi_{x}(y)} \leq \frac{C_{0}\exp (a)}{d(x,y)} \cdot$$ 
\end{lemma}

\begin{proof}
Let $y \in \mathcal{Y}$ and assume that $d(x,y)<a$. It is easy to check that $$\frac{\langle b_{x}(y),\chi_{A}\rangle}{\varphi_{x}(y)} \leq \frac{\exp (a)}{d(x,y)}\cdot $$
Now assume that $d(x,y)\geq a$.

If $v\in A(a)$ and $w\in O_{R}(x,y)$, since $\mathcal{O}_{R}(x,y)\cap A(a)=\varnothing$ we have $d_{x}(v,w)>\exp{(-a)}$. 

Using the first inequality Lemma \ref{ineghch} and the above observation we have for all $w\in \mathcal{O}_{R}(x,y)$:
\begin{align*}
\langle b_{x}(y),\chi_{A} \rangle &\leq\exp\bigg({-\frac{\alpha}{2}d(x,y)}\bigg)\int_{A(a)}  \frac{1}{d^{\alpha}_{x}(v,w)} \exp{\bigg(\alpha(R+\delta)\bigg)}d\mu_{x}(v) \\
&\leq \exp{\bigg(\alpha(R+\delta+a)\bigg)}\|\mu\|_{x}\exp\bigg({-\frac{\alpha}{2}d(x,y)}\bigg)\\
&\leq \exp{\bigg(\alpha(R+\delta+a)\bigg)}\|\mu\|_{x}\frac{\varphi_{x}(y)}{Q_{1}\big(d(x,y)\big)}
\end{align*}
where the last inequality comes from the left hand side of Harish-Chandra estimates on $\mathcal{Y}$. Since $Q_{1}$ is a polynomial of degree one, the proof is complete. 
\end{proof}

Now assume that $\varphi_{x}$ satisfies the left hand side of  Harish-Chandra estimates on $\mathcal{Y}=\Gamma x$.

\begin{prop}\label{deform}
Let $\psi_{t}\in l^{1}(\Gamma)$ such that $\|\psi_{t}\|_{1}\leq 1$, and assume that $$\lim_{t\rightarrow +\infty}\psi_{t}(\gamma)=0,$$ for all $\gamma \in \Gamma$. Then for every Borel subset $A\subset  \partial X$ we have for all $a>0$ $$\limsup_{t \rightarrow+\infty} \sum_{\gamma \in \Gamma}\psi_{t}(\gamma)\frac{\langle \pi_{x}(\gamma)\textbf{1}_{\partial X},\chi_{A}\rangle}{\phi_{x}(\gamma)} \leq \limsup_{t \rightarrow+\infty} \sum_{\gamma \in \Gamma}\psi_{t}(\gamma)D_{\gamma  x}  ( \chi_{C_{R}(x,A(a))}).$$
\end{prop}
\begin{proof}
Let $A$ be Borel subset of $\partial X$ and let $a$ be a positive number. Let $t_{0}$ be another positive real number. Consider the following partition of $\Gamma$: $$\Gamma=\Gamma_{1}\sqcup \Gamma_{2}\sqcup  \Gamma_{2}$$ with $$\Gamma_{1}=\lbrace\gamma \in \Gamma | d(x,\gamma  x)\leq t_{0} \rbrace$$ and  $$\Gamma_{2}=\lbrace \gamma \in \Gamma | \mathcal{O}_{R}(x,\gamma  x)\cap A(a)\neq \varnothing \rbrace \cap \Gamma_{1}^{c}$$ and  $$\Gamma_{3}=\lbrace \gamma \in \Gamma | \mathcal{O}_{R}(x,\gamma  x)\cap A(a)= \varnothing \rbrace \cap \Gamma_{1}^{c}.$$   

Since $\pi_{x}$ is positive, we have that $$\sum_{\Gamma_{1}}\psi_{t}(\gamma)\frac{\langle \pi_{x}(\gamma)\textbf{1}_{\partial X},\chi_{A}\rangle}{\phi_{x}(\gamma)} \leq \sum_{\Gamma_{1}}\psi_{t}(\gamma).$$

Observe that $$\gamma \in \Gamma_{2}\Leftrightarrow D_{\gamma  x}(\chi_{C_{R}(x,A(a))})=1.$$ Thus 
$$\sum_{\gamma \in \Gamma_{2}}\psi_{t}(\gamma)\frac{\langle \pi_{x}(\gamma)\textbf{1}_{\partial X},\chi_{A}\rangle}{\phi_{x}(\gamma)} \leq \sum_{\gamma \in \Gamma_{2} }\psi_{t}(\gamma) D_{\gamma  x}(\chi_{C_{R}(x,A(a))}). $$

Observe that $$\langle b_{x}(\gamma x),\chi_{A}\rangle=\langle \pi_{x}(\gamma)\textbf{1}_{\partial X},\chi_{A}\rangle.$$

Since $\mathcal{Y}=\Gamma x$ we can apply Lemma \ref{majorationendehors} via the above observation and thus for all $t_{0}>0$: $$\sum_{\Gamma_{3}} \psi_{t}(\gamma)\frac{\langle \pi_{x}(\gamma)\textbf{1}_{\partial X},\chi_{A}\rangle}{\phi_{x}(\gamma)}  \leq \bigg( \sum_{\Gamma}\psi_{t} (\gamma)  \bigg)C_{0}\frac{\exp(a)}{t_{0}}.$$ Then, since $\|\psi_{t}\|_{1}\leq 1$, we obtain
for all $t_{0}>0$ $$\sum_{\Gamma_{3}} \psi_{t}(\gamma)\frac{\langle \pi_{x}(\gamma)\textbf{1}_{\partial X},\chi_{A}\rangle}{\phi_{x}(\gamma)}  \leq C_{0}\frac{\exp(a)}{t_{0}}.$$

It follows that for all $a>0$ and for all $t>t_{0}$ we have 
$$
\sum_{\gamma \in \Gamma}\psi_{t}(\gamma)\frac{\langle \pi_{x}(\gamma)\textbf{1}_{\partial X},\chi_{A}\rangle}{\phi_{x}(\gamma)}\leq \sum_{\Gamma_{1}}\psi_{t}(\gamma)+\sum_{\Gamma}\psi_{t}(\gamma)D_{\gamma  x}(\chi_{C_{R}(x,A(a))})+C_{0}\frac{\exp(a)}{t_{0}}  .
$$

Since $\psi_{t}(\gamma)\rightarrow 0$ as $t\rightarrow +\infty $, we obtain by taking the $\limsup$ in the above inequality 
$$ \limsup_{t\rightarrow +\infty}\sum_{\gamma \in \Gamma}\psi_{t}(\gamma)\frac{\langle \pi_{x}(\gamma)\textbf{1}_{\partial X},\chi_{A}\rangle}{\phi_{x}(\gamma)}\leq \limsup_{t\rightarrow +\infty} \sum_{\gamma \in \Gamma }\psi_{t}(\gamma) D_{\gamma  x}(\chi_{C_{R}(x,A(a))})+C_{0}\frac{\exp(a)}{t_{0}} .$$
This inequality holds for all $t_{0}>0$, so we take $t_{0}\rightarrow +\infty$ and the proof is complete.
\end{proof}

\subsection{A consequence of Roblin's Theorem}\label{BMSroblin}
 If $A\subset \partial X$, we denote by $\partial A$ its frontier.
 We need a consequence of Theorem \ref{roblin} which counts the points of a $\Gamma$-orbit $\Gamma x$ in $C_{R}(x,A)$ when $A$ is a Borel subset with $\mu_{x}(\partial A)=0$. This is based on the regularity of the conformal densities. 
We recall that the topology of $\overline{X}$ is compatible with the metric topology defined on $\partial X$ by the visual metrics (see \cite[\S 1.5]{Bou}). If $O\subset \overline{X}$, we denote by $\overline{O}$ its closure in $\overline{X}$.

First, observe the following:
\begin{lemma}\label{lemmferme}
Let $A$ be a closed subset of $\partial X$. Then $\overline{C_{R}(x,A)}=C_{R}(x,A)\sqcup A$ .
\end{lemma}
\begin{proof}
It is easy to check that $C_{R}(x,A)\cup A \subset \overline{C_{R}(x,A)}$.

Now, assume that $v \in \overline{C_{R}(x,A)}\cap \partial X$ (otherwise there is nothing to do). We shall prove that $v \in A$. There exists a sequence of $y_{n} \in C_{R}(x,A)$ such that $y_{n} \rightarrow v$ with respect to  the topology of $\overline{X}$. Since $y_{n}$ is in  $C_{R}(x,A)$, there exists $v_{n} \in A\cap\mathcal{O} _{R}(x,y_{n})$ such that $(y_{n},v_{n})_{x}\geq d(x,y_{n})-R$,  for all integers $n$. Thus, we have 
\begin{align*}
(v_{n},v)_{x}&\geq \min{\lbrace(v_{n},y_{n})_{x},(y_{n},v)_{x} \rbrace}-\delta \\
& \geq (y_{n},v)_{x}-R-\delta.
\end{align*}
where the last inequality follows from $(y_{n},v)_{x}\leq d(x,y_{n})$.
Since $y_{n} \rightarrow v$, it follows that $(y_{n},v)_{x}$ goes to $+\infty$, and so $v_{n}\rightarrow v$ with respect to  $d_{x}$. Since $A$ is closed the proof is done.
 \end{proof}
Then we shall give a proof the following corollary:
\begin{coro}\label{cororoblin}(Extracted from \cite[Th\'eor\`eme 4.1.1, Chapitre 4]{Ro})
Let $\Gamma$ be a discrete group of isometries of $X$ with a non-arithmetic spectrum. Assume that $\Gamma$ admits a finite BMS measure associated with a $\Gamma$-invariant conformal density $\mu$ of dimension $\alpha=\alpha(\Gamma)$. Let $A,B$ be two Borel subsets such that $\mu_{x}(\partial A)=0=\mu_{x}(\partial B)$. Then for all $\rho>0$ we have
$$ \limsup_{n \rightarrow +\infty}\frac{1}{|C_{n}(x,\rho)|}\sum_{\gamma \in C_{n}(x,\rho)}D_{\gamma^{-1}  x } \otimes D_{\gamma  x} (\chi_{C_{R}(x,A)} \otimes \chi_{C_{R}(x,B)}) \leq  \frac{\mu_{x}(A)\mu_{x}(B)}{\| \mu \|_{x}^{2}}.$$ 
\end{coro}

 \begin{proof}
 Let $\rho$ be a positive real number.
 We have for all $n$ large enough:
\begin{align*}
 \frac{1}{|C_{n}(x,\rho)|} &\sum_{ \gamma \in C_{n}(x,\rho)} D_{\gamma^{-1}  x} \otimes D_{\gamma  x} 
= \frac{\alpha \|m_{\Gamma} \| \exp(-\alpha (n+\rho))}{|C_{n}(x,\rho)|\alpha \|m_{\Gamma}\| \exp(-\alpha (n+\rho))}  \sum_{ \gamma \in \Gamma_{n+\rho}(x)} D_{\gamma^{-1}  x} \otimes D_{\gamma  x} \\ 
 &-  \frac{\alpha \|m_{\Gamma}\| \exp(-\alpha (n-\rho))}{|C_{n}(x,\rho)|\alpha \|m_{\Gamma}\| \exp(-\alpha (n-\rho))}  \sum_{ \gamma \in \Gamma_{n-\rho}(x)} D_{\gamma^{-1}  x} \otimes D_{\gamma  x} .
  \end{align*}
  
  The estimation (\ref{estimcouronne}) for annulii implies as $n\rightarrow +\infty $  $$|C_{n}(x,\rho)|\alpha \|m_{\Gamma}\| \exp(-\alpha (n+\rho))  \sim 2\sinh(\alpha \rho)\exp{(-\alpha \rho)}\|\mu \|^{2}_{x} $$ and 
  $$ |C_{n}(x,\rho)|\alpha \|m_{\Gamma}\| \exp(-\alpha (n-\rho)) \sim 2\sinh(\alpha \rho)\exp{(\alpha \rho)}\|\mu \|^{2}_{x} .$$
Therefore Theorem \ref{roblin} implies
\begin{equation}\label{conseq1roblin}
 \frac{1}{|C_{n}(x,\rho)|} \sum_{ \gamma \in C_{n}(x,\rho)} D_{\gamma^{-1}  x} \otimes D_{\gamma  x} \rightharpoonup
 \frac{1}{\|\mu_{x} \|^{2}}\mu_{x} \otimes \mu_{x},
  \end{equation}
with respect to  the weak* topology of $C(\overline{X}\times \overline{X})^{*}$.

Consider a Borel subset $A$ of $\partial X$ such that $\mu_{x}(\partial A)=0$. 
We have $\mu_{x}(\overline{A})=\mu_{x}(A).$ Thus, by Lemma \ref{lemmferme} we obtain
\begin{align*}
\mu_{x}(\overline{C_{R}(x,\overline{A})})&=\mu_{x}(A).
\end{align*} 

  Let $\epsilon>0$. Since $\mu_{x}$ is a regular measure there exists an open subset $O_{A}$ of $\overline {X}$ such that 

 \begin{equation}\label{mesurtopo}
\overline{C_{R}(x,\overline{A})} \subset O_{A} \mbox{ and } \mu_{x}(O_{A}) \leq \mu_{x}(A)+ \epsilon.
 \end{equation}
 
The subset $\overline{C_{R}(x,\overline{A})}$ is a compact subset of $\overline{X}$.
By Urysohn's lemma, we can find a compactly supported function $f_{O_{A}}$ such that  $$ \chi_{\overline{C_{R}(x,\overline{A})}}\leq f_{O_{A}}\leq \chi_{O_{A}}.$$

Let $B$ be another Borel subset such that $\mu_{x}(\partial B)=0$.  Let $f_{O_{B}}$ be the continuous function given by the above construction. Notice that for  all $n$ we have:
$$\sum_{\gamma \in C_{n}(x,\rho)}D_{\gamma x} \otimes D_{\gamma^{-1} x}(\chi_{C_{r}(x,A)}\otimes \chi_{C_{r}(x,B)})\leq\sum_{\gamma \in C_{n}(x,\rho)}D_{\gamma x} \otimes D_{\gamma^{-1} x}(f_{O_{A}}\otimes f_{O_{B}}). $$

 The consequence of Roblin's theorem (\ref{conseq1roblin}) implies:
\begin{align*}
&\limsup_{n\rightarrow \infty}\frac{\| \mu_{x}\|^{2}}{|C_{n}(x,\rho)|}\sum_{\gamma \in C_{n}(x,\rho)}D_{\gamma x} \otimes D_{\gamma^{-1} x}(\chi_{C_{R}(x,A)}\otimes \chi_{ C_{R}(x,B)}) \\
&\leq \limsup_{n\rightarrow \infty}\frac{\| \mu_{x}\|^{2}}{|C_{n}(x,\rho)|}\sum_{\gamma \in C_{n}(x,\rho)}D_{\gamma x} \otimes D_{\gamma^{-1} x}(f_{O_{A}}\otimes f_{O_{B}})\\
&=\lim_{n\rightarrow \infty}\frac{\| \mu_{x}\|^{2}}{|C_{n}(x,\rho)|}\sum_{\gamma \in C_{n}(x,\rho)}D_{\gamma x} \otimes D_{\gamma^{-1} x}(f_{O_{A}}\otimes f_{O_{B}})\\
&=\int_{\partial X \times \partial X}(f_{O_{A}}\otimes f_{O_{B}}) d \mu_{x} \otimes d\mu_{x}\\
&\leq \mu_{x}(A)\mu_{x}(B)+\epsilon(\mu_{x}(A)+\mu_{x}(B))+\epsilon^{2},
\end{align*}
where the last inequality follows from (\ref{mesurtopo}).
The above inequality holds for all $\epsilon>0$, so the proof is done.
 \end{proof}

\subsection{An application of Roblin's equidistribution Theorem}
Let $\rho>0$, and let $N_{x,\rho}$ be an integer such that for all $n \geq N_{x,\rho}$ the sequence $\mathcal{M}_{x,\rho}^{n}$ is well defined. The purpose of this section is to use Corollary \ref{cororoblin} for computing the limit of the sequence of operator-valued measures $(\mathcal{M}_{x,\rho}^{n})_{n \geq N_{x,\rho}}$.\\
 We assume that $\varphi_{x}$ satisfies the left hand side of Harish-Chandra estimates on $\Gamma x$.

\begin{prop}\label{applicationroblin}
Let $A,B,U \subset \partial {X}$ be Borel subsets such that $
\mu_{x}(\partial A)=\mu_{x}(\partial B)=\mu_{x}(\partial U)=0$, let $\widehat{U}=C_{R}(x,U)\cup U$ be a borel subset of $\overline{X}$. Then we have:
$$\lim_{n\rightarrow +\infty} \langle \mathcal{M}^{n}_{x,\rho}(\chi_{\widehat{U}})\chi_{A},\chi_{B} \rangle=\frac{\mu_{x}(U \cap B)\mu_{x}(A)}{ \| \mu_{x} \|^{2}}.$$

\end{prop}

We need some lemmas to prepare the proof of this proposition.

\begin{lemma}\label{lemprepar1}
Let $B,U \subset \partial {X}$ be Borel subsets such that $
\mu_{x}(\partial B)=\mu_{x}(\partial U)=0$, let $\widehat{U}=C_{R}(x,U)\cup U$ be a borel subset of $\overline{X}$ satisfying $U\cap B(b)=\varnothing$, for some $b>0$. Then we have $$\limsup_{n \rightarrow +\infty} \langle \mathcal{M} ^{n}_{x,\rho} (\chi_{\widehat{U}} ) \textbf{1}_{\partial X} , \chi_{B} \rangle =0.$$
\end{lemma}

\begin{proof}
For all $n\geq N_{x,\rho}$ we have:
\begin{eqnarray*}
 \langle \mathcal{M} ^{n}_{x,\rho} (\chi_{\widehat{U}} ) \textbf{1}_{\partial X} , \chi_{B} \rangle &=&\frac{1}{ | C_{n}(x,\rho) |} \sum_{\gamma \in C_{n}(x,\rho)} D_{\gamma  x} (\chi_{\widehat{U}}) \frac{\langle \pi_{x}(\gamma)\textbf{1}_{\partial X},\chi_{B} \rangle}{\phi_{x}(\gamma)} \\
 & = &\sum_{\gamma \in \Gamma} \psi_{n}(\gamma)\frac{ \langle \pi_{x}(\gamma)\textbf{1}_{\partial X},\chi_{B} \rangle}{\phi_{x}(\gamma)}
\end{eqnarray*}

where the inequality follows from the fact that $\pi_{x}$ is positive, and where $$\psi_{n}(\gamma):= \frac{1}{ |C_{n}(x,\rho)| } \chi_{C_{n}(x,\rho)}(\gamma) D_{\gamma  x} (\chi_{\widehat{U}}).$$

Proposition \ref{deform} implies that 
\begin{eqnarray*}
\limsup_{n \rightarrow + \infty} \langle \mathcal{M} ^{n}_{x,\rho} (\chi_{\widehat{U}} ) \textbf{1}_{\partial X} , \chi_{B} \rangle &\leq& \limsup_{n \rightarrow +\infty} \sum_{\gamma \in \Gamma} \psi_{n}(\gamma)D_{\gamma  x} (\chi_{(\widehat{U})})D_{\gamma  x}(\chi_{C_{R}(x,B(b))}) \\
&=&\limsup_{n \rightarrow +\infty}  \frac{1}{ | C_{n}(x,\rho) |} \sum_{\gamma \in C_{n}(x,\rho)} D_{\gamma  x} (\chi_{\widehat{U}\cap C_{R}(x,B(b))})\\
&\leq &\limsup_{n \rightarrow +\infty}  \frac{1}{ | C_{n}(x,\rho) |} \sum_{\gamma \in C_{n}(x,\rho)} D_{\gamma  x} (\chi_{ C_{R}(x,U\cap B(b))})
 \end{eqnarray*}
 Note the general fact $\partial (A\cap B)\subset \partial A \cup \partial B$.
 Corollary \ref{cororoblin} implies that
 $$\limsup_{n \rightarrow +\infty} \langle \mathcal{M} ^{n}_{x,\rho} (\chi_{\widehat{U}} ) \textbf{1}_{\partial X} , \chi_{B} \rangle \leq \frac{\mu_{x}\big(U\cap B(b)\big)}{\|\mu_{x}\|} \cdot$$
  By hypothesis  $U\cap B(b)=\varnothing$ thus we have $$ \limsup_{n \rightarrow +\infty} \langle \mathcal{M} ^{n}_{x,\rho} (\chi_{\widehat{U}} ) \textbf{1}_{\partial X} , \chi_{B} \rangle =0.$$
 \end{proof}

\begin{lemma}\label{lemprepar2}
Let $\widehat{U}$ be a Borel subset of $\overline{X}$ and let $A$ be a Borel subset of $\partial X$. We have $$\limsup_{n\rightarrow +\infty} \langle \mathcal{M} ^{n}_{x,\rho} (\chi_{\widehat{U} } )  \chi_{A} ,\textbf{1}_{\partial X} \rangle \leq \limsup_{n \rightarrow +\infty} \frac{1}{ | C_{n}(x,\rho) |} \sum_{\gamma \in C_{n}(x,\rho)} D_{\gamma^{-1} x } (\chi_{\widehat{U}})  D_{\gamma  x} (\chi_{C_{R}(x,A(a))}).$$ 
 \end{lemma}

\begin{proof}
We have for all $n\geq N_{x,\rho}$:
\begin{eqnarray*}
\langle \mathcal{M} ^{n}_{x,\rho} (\chi_{\widehat{U}} )  \chi_{A} ,\textbf{1}_{\partial X} \rangle 
&=& \sum_{\gamma \in \Gamma} \psi_{n}(\gamma)\frac{\langle \pi_{x}(\gamma) \textbf{1}_{\partial X} ,\chi_{A} \rangle}{\phi_{x}(\gamma)},
\end{eqnarray*}
with $$\psi_{n}(\gamma)=\frac{1}{ | C_{n}(x,\rho) |}\chi_{C_{n}(x,\rho)}(\gamma) D_{\gamma ^{-1} x}(\chi_{\widehat{U}}).$$

Applying Proposition \ref{deform} to $\psi_{n}$ defined above we obtain that: 
$$\limsup_{n\rightarrow +\infty} \langle \mathcal{M} ^{n}_{x,\rho} (\chi_{\widehat{U} } )  \chi_{A} ,\textbf{1}_{\partial X} \rangle \leq \limsup_{n \rightarrow +\infty} \frac{1}{ | C_{n}(x,\rho) |} \sum_{\gamma \in C_{n}(x,\rho)} D_{\gamma^{-1} x } (\chi_{\widehat{U}})  D_{\gamma  x} (\chi_{C_{R}(x,A(a))}).$$
\end{proof}

\begin{lemma}\label{lemineg}
Let  $A,B,U \subset \partial {X}$ be Borel subsets such that $\mu_{x}(\partial A)=\mu_{x}(\partial B)=\mu_{x}(\partial U)=0$ and let $\widehat{U}=C_{R}(x,U)\cup U$ be a Borel subset of $\overline{X}$ with $\mu_{x}(\partial U)=0$. 

$$\limsup_{n\rightarrow +\infty} \langle \mathcal{M}^{n}_{x,\rho}(\chi_{\widehat{U}})\chi_{A},\chi_{B} \rangle\leq \frac{\mu_{x}(U \cap B)\mu_{x}(A)}{\| \mu_{x} \|^{2}}\cdot $$
\end{lemma}
\begin{proof}
Let $a>0$ and $b>0$, and consider $A(a)$ and $B(b)$ such that $\mu_{x}(\partial B(b))=0=\mu_{x}(\partial A(a))$. Let $B(b)^{c}=\partial X \backslash B(b)$. Set $\widehat{U_{1}}=\widehat{U} \cap \overline{C_{R}(x,B(b))}$ and  $\widehat{U_{2}}=\widehat{U} \cap \overline{X} / \overline{C_{R}(x,B(b))}$. Let  $U_{1}=\widehat{U_{1}}\cap \partial X$ and  $U_{2}=\widehat{U_{2}}\cap \partial X$ and notice that $U_{1}=U\cap \overline{B(b)}$ and $U_{2}=U\cap\partial X / \overline{B(b)}$. Observe that $U_{2}\cap B(b)=\varnothing$. Since $\widehat{U}=\widehat{U_{1}}\sqcup \widehat{U_{2}}$ we have:
\begin{eqnarray*}
\langle \mathcal{M}_{x,\rho}^{n}(\chi_{\widehat{U} })\chi_{A},\chi_{B}\rangle&=&
 \langle \mathcal{M}_{x,\rho}^{n}(\chi_{\widehat{U_{1}}})\chi_{A},\chi_{B} \rangle+\langle \mathcal{M}_{x,\rho}^{n}(\chi_{\widehat{U_{2}}})\chi_{A},\chi_{B} \rangle\\
 &\leq & \langle \mathcal{M}_{x,\rho}^{n}(\chi_{\widehat{U_{1}}})\chi_{A},\textbf{1}_{\partial X} \rangle+\langle \mathcal{M}_{x,\rho}^{n}(\chi_{\widehat{U_{2}}})\textbf{1}_{\partial X},\chi_{B} \rangle.
\end{eqnarray*}
 Applying Lemma \ref{lemprepar1} to the second term and Lemma \ref{lemprepar2} to the first term of the right hand side above inequality, we obtain: $$\limsup_{n \rightarrow +\infty} \langle \mathcal{M}_{x,\rho}^{n}(\chi_{\widehat{U} })\chi_{A},\chi_{B}\rangle \leq \limsup_{n \rightarrow +\infty} \frac{1}{ | C_{n}(x,\rho) |} \sum_{\gamma \in C_{n}(x,\rho)}D_{\gamma ^{-1} x} (\chi_{\widehat{U_{1}}})  D_{\gamma  x} (\chi_{C_{R}(x,A(a))} ).$$ 
 Since $\mu_{x}(\partial U_{1})=0=\mu_{x}(\partial A(a))$, Roblin's corollary \ref{cororoblin} leads to 
  $$\limsup_{n \rightarrow +\infty} \langle \mathcal{M}_{x,\rho}^{n}(\chi_{\widehat{U} })\chi_{A},\chi_{B}\rangle \leq \frac{\mu_{x}(U\cap B(b))\mu_{x}(A(a))}{\|\mu_{x}\|^{2}}.$$
 Because the above inequality holds for all $a,b>0$ but at most countably many values of $a$ and $b$, we obtain the required inequality.
 
\end{proof}

\begin{proof}[Proof of Proposition \ref{applicationroblin}]
By Lemma \ref{lemineg} it is sufficient to prove that $$\liminf_{n\rightarrow +\infty} \langle \mathcal{M}^{n}_{x,\rho}(\chi_{\widehat{U}})\chi_{A},\chi_{B} \rangle=\frac{\mu_{x}(U \cap B)\mu_{x}(A)}{\| \mu_{x}\|^{2}}\cdot $$
If $W$ is a Borel subset of $\partial X$ (or $\overline{X}$), we set $W^{0}=W$ and $W^{1}=\partial X \backslash W$ (or $W^{1}=\overline{X} \backslash W$). 
We have 
\begin{eqnarray*}
1&=&\langle \mathcal{M}_{x,\rho}^{n}(\textbf{1}_{\overline{ X}})\textbf{1}_{\partial X},\textbf{1}_{\partial X}\rangle \\
&=&\langle \mathcal{M}_{x,\rho}^{n}(\chi_{\widehat{U}^0}+\chi_{\widehat{U}^1})\chi_{A^0}+\chi_{A^1},\chi_{B^0}+\chi_{B^1} \rangle \\
 &=& \sum_{i,j,k} \langle\mathcal{M}_{x,\rho}^{n}(\chi_{\widehat{U}^i})\chi_{A^j},\chi_{B^k} \rangle \\
 &=&  \langle\mathcal{M}_{x,\rho}^{n}(\chi_{\widehat{U}})\chi_{A},\chi_{B} \rangle+\sum_{i,j,k \neq (0,0,0)} \langle\mathcal{M}_{x,\rho}^{n}(\chi_{\widehat{U}^i})\chi_{A^j},\chi_{B^k} \rangle. 
\end{eqnarray*}
Then 
\begin{eqnarray*}
1 &\leq& \liminf_{n \rightarrow +\infty}\langle \mathcal{M}_{x,\rho}^{n}(\chi_{\widehat{U}})\chi_{A},\chi_{B} \rangle+\sum_{i,j,k \neq (0,0,0)}  \limsup_{n \rightarrow + \infty} \langle \mathcal{M}_{x,\rho}^{n}(\chi_{\widehat{U}^i})\chi_{A^j},\chi_{B^k} \rangle \\
&\leq &  \limsup_{n \rightarrow +\infty}\langle \mathcal{M}_{x,\rho}^{n}(\chi_{\widehat{U}})\chi_{A},\chi_{B} \rangle+\sum_{i,j,k \neq (0,0,0)}  \limsup_{n \rightarrow + \infty} \langle \mathcal{M}_{x,\rho}^{n}(\chi_{\widehat{U}^i})\chi_{A^j},\chi_{B^k} \rangle \\
& \leq & \frac{1}{\| \mu_{x} \|^{2}}\sum_{i,j,k } \mu_{x}(U^{i} \cap B^k) \mu_{x}(A^j) \\
&=& 1,
\end{eqnarray*}
where the last inequality comes from Lemma \ref{lemineg}. Hence the inequalities of the above computation are equalities, so $$\liminf_{n\rightarrow +\infty} \langle \mathcal{M}^{n}_{x,\rho}(\chi_{\widehat{U}})\chi_{A},\chi_{B} \rangle =\frac{\mu_{x}(U \cap B)\mu_{x}(A)}{\| \mu_{x}\|^{2}}=\limsup_{n\rightarrow +\infty} \langle \mathcal{M}^{n}_{x,\rho}(\chi_{\widehat{U}})\chi_{A},\chi_{B} \rangle $$
and the proof is done.

\end{proof}
\section{Conclusion}\label{conclusion}

\subsection{Standard facts about Borel subsets of measure zero frontier}
Let us recall two standard facts of measure theory that we state as lemmas:
\begin{lemma}\label{mesurenullebord}
Assume that $(Z,d,\mu)$ is a metric measure space. Then the $\sigma$-algebra generated by Borel subset with measure zero frontier generates the Borel $\sigma$-algebra.
\end{lemma}

Let $\chi_{A}$ be the characteristic function of a Borel subset $A$ of $\partial X$.
We state another useful lemma (see \cite[Appendix B, Lemma B.2 (1)]{BM} for a proof):

\begin{lemma}\label{denseL2}
Assume that $(Z,d,\mu)$ is a metric measure space where $\mu$ is Radon measure. Then the closure of the subspace spanned by the characteristic functions of Borel subset having zero measure frontier is $$\overline{Span\lbrace \chi_{A} \mbox{ such that } \mu(\partial A)=0 \rbrace}^{L^2}=L^{2}(Z,\mu).$$
\end{lemma}

\subsection{Proofs}

\begin{proof}[Proof of Theorem A] 
Let $\mu$ be a $\Gamma$-invariant conformal density of dimension $\alpha(\Gamma)$, where $\Gamma$ in $\mathcal{C}$. Since for all $x\in X$, the metric measure space $(\Lambda_{\Gamma},d_{x},\mu_{x})$ is Ahlfors $\alpha$-regular Proposition \ref{H-CHestimates} ensures that the Harish-Chandra estimates hold on $\Gamma x$. Hence Proposition  \ref{inegsample} and \ref{applicationroblin} are available. The sequence $\mathcal{M}_{x,\rho}^{n}$ is defined for $n\geq N_{x,\rho}$ for some integer $N_{x,\rho}$.
 There are two steps. \\

\textbf{Step 1: $(\mathcal{M}_{x,\rho}^{n})_{n\geq N_{x,\rho}}$ is uniformly bounded.} First of all, observe that  $\mathcal{M}_{x,\rho}^{n}(\textbf{1}_{ \overline{X} })$ is self-adjoint (see Proposition \ref{prop123} (1)). Note that $\mathcal{M}_{x,\rho}^{n}(\textbf{1}_{ \overline{X} })$ preserves $L^{\infty}(\partial X,\mu_{x})$, and by duality it preserves also $L^{1}(\partial X,\mu_{x})$.

Combining Proposition \ref{inegsample} with the fact that $\mathcal{M}_{x,\rho}^{n}(\textbf{1}_{ \overline{X} })\textbf{1}_{\partial X}=F_{x,\rho}^{n}$, we have that the sequence $\big(\mathcal{M}_{x,\rho}^{n}(\textbf{1}_{ \overline{X} })\big)_{n\geq N_{x,\rho}}$, with $\mathcal{M}_{x,\rho}^{n}(\textbf{1}_{ \overline{X} })$ viewed as operators from $L^{\infty}(\partial X,\mu_{x})$ to $L^{\infty}(\partial X,\mu)$, is uniformly bounded. Riesz-Thorin interpolation theorem implies the sequence $(\mathcal{M}_{x,\rho}^{n}(\textbf{1}_{ \overline{X} })\big)_{n\geq N_{x,\rho}}$, with $\mathcal{M}_{x,\rho}^{n}(\textbf{1}_{ \overline{X} })$ viewed as operators in $\mathcal{B} \big( L^{2}(\partial X, \mu_{x}) \big)$, is uniformly bounded.
 Then Proposition \ref{prop123} (2) completes \textbf{Step 1}.
\\

\textbf{Step 2: computation of the limit of $(\mathcal{M}_{x,\rho}^{n})_{n \geq N_{x,\rho}}$}. By the Banach-Alaoglu theorem, \textbf{Step 1} implies that  $(\mathcal{M}_{x,\rho}^{n})_{n\geq N_{x,\rho}}$ has accumulation points. Let $\mathcal{M}^{\infty}_{x}$ be an accumulation point of $(\mathcal{M}_{x,\rho}^{n})_{n \geq N_{x,\rho}}$ with respect to  the weak* topology of $\mathcal{L}\big( C(\overline{X}),\mathcal{B}(L^{2}(\partial X,\mu_{x}))\big)$. Let $\widehat{U}=C_{R}(x,U)\cup U$ be Borel subset of $\overline{X}$ with $U$ be a Borel subset of $\partial X$ such that $\mu_{x}(\partial U)=0$. It follows from Proposition \ref{applicationroblin} and from the definition (\ref{mesurelimite}) of $\mathcal{M}_{x}$ that for all Borel subsets $A,B\subset \partial X$ with $\mu_{x}(\partial A)=\mu_{x}(\partial B)=0$ we have that $$\langle \mathcal{M}^{\infty}_{x}(\chi_{\widehat {U}})\chi_{A},\chi_{B}\rangle=\frac{\mu_{x}(U \cap B)\mu_{x}(A)}{\| \mu_{x}\|^{2}}=\langle \mathcal{M}_{x}(\chi_{\widehat {U}})\chi_{A},\chi_{B}\rangle.$$
The above equality holds for all open balls $B_{X}(x,r)$ of $X$, namely  $$\langle \mathcal{M}^{\infty}_{x}(\chi_{B_{X}(x,r)})\chi_{A},\chi_{B}\rangle=\frac{\mu_{x}(B_{X}(x,r) \cap B)\mu_{x}(A)}{\| \mu_{x}\|^{2}}=\langle \mathcal{M}_{x}(\chi_{B_{X}(x,r)})\chi_{A},\chi_{B}\rangle=0.$$
 Since the open balls $\lbrace B_{X}(x,r),x\in X,r>0 \rbrace $ of $X$ together with the subsets $\widehat{U}=C_{R}(x,U)\cup U$, with $U$ Borel subsets of $\partial X$ such that $\mu_{x}(\partial U)=0$ generate the Borel $\sigma$-algebra of $\overline{X}$, Carath\'eodory's extension theorem implies that for all $f\in C(\overline{X})$ and for all Borel subsets $A,B\subset \partial X$ satisfying $\mu_{x}(\partial A)=\mu_{x}(\partial B)=0$ we have $$\langle \mathcal{M}^{\infty}_{x}(f)\chi_{A},\chi_{B}\rangle=\langle \mathcal{M}_{x}(f)\chi_{A},\chi_{B}\rangle.$$
Lemma \ref{denseL2} combined with the above equality imply that the operators $\mathcal{M}^{\infty}_{x}$ and $\mathcal{M}_{x}$ regarded as functionals of $(C(\overline{X}) \widehat{\otimes} L^{2}(\partial X, \mu_{x}) \widehat{\otimes} L^{2}(\partial X, \mu_{x}))^{*}$ (see (\ref{iso3})) are equal on a dense subset of $C(\overline{X}) \widehat{\otimes} L^{2}(\partial X, \mu_{x}) \widehat{\otimes} L^{2}(\partial X, \mu_{x})$. We deduce that $\mathcal{M}_{x}$ is the unique accumulation point of the sequence $(\mathcal{M}_{x,\rho}^{n})_{n \geq N_{x,\rho}}$.
\end{proof}

\begin{proof}[Proof of Corollary B]
Apply the definition of weak$^{*}$ convergence to $\textbf{1}_{\overline{X}}\otimes  \xi \otimes \eta$ for all $\xi,\eta \in L^{2}(\partial X, \mu_{x})$, and observe that $\|\mu\|_{x}^{2}\mathcal{M}_{x}(\textbf{1}_{\overline{X}})$ is the orthogonal projection onto the space of constant functions.
\end{proof}

\begin{proof}[Proof of Corollary C]
Since $(\pi_{x})_{x\in X}$ are unitarily equivalent, it suffices to prove irreducibility for some $\pi_{x}$ with $x$ in $ X$. Theorem A shows that the vector $\textbf{1}_{\partial X}$ is cyclic for the representation $\pi_{x}$.
Moreover, Corollary B shows that the orthogonal projection onto the space of constant functions is in the von Neumann algebra associated with $\pi_{x}$. Then, a standard argument (see for example \cite[Lemma 6.1]{LG}) completes the proof. 

\end{proof}

\begin{remark}
The hypothesis:\emph{ $\Gamma$ is convex cocompact or a lattice in a rank one semisimple Lie group} guarantees the Ahlfors regularity of the limit set, that implies the Harish-Chandra estimates of $\varphi_{x}$ on $CH(\Lambda_{\Gamma})$ and on $\Gamma x$. In other words, the proof of irreducibility of boundary representations for a geometrically finite group with a non-arithmetic spectrum is reduced, by this approach, to the Harish-Chandra estimates of $\varphi_{x}$ for each $x\in  X$ on $CH(\Lambda_{\Gamma}) \backslash B_{X}(x,R_{x})$ and on the orbit $\Gamma x \backslash B_{X}(x,R_{x})$ for some $R_{x}>0$. And this approach should apply to  some geometrically finite groups which are neither convex cocompact and nor lattices.
\end{remark}
\section{Some remarks about equidistribution results}\label{equidistribution}

 \subsection{Dirac-Weierstrass family}\label{DW}
Let $\Gamma$ be a discrete group of isometries of $X$.
Consider $(d_{x})_{x\in X}$ a visual metric on $\partial X$, and let $\mu$ be a $\Gamma$-invariant conformal densitiy of dimension $\alpha$. We fix $x\in X$ and we follow \cite[Chapter 2, \S 2.1, p\ 46]{Jo}, and adapt the definition of a Dirac-Weierstrass family to the density $\mu$:
\begin{defi}
\label{diracweier} A Dirac-Weierstrass family $(K(y,\cdot))_{y\in X}$ with respect to  $\mu_{x}$, is a continuous map
 $K: (y,v) \in X\times\partial X \mapsto K(y,v) \in \mathbb{R}$  satisfying
\begin{enumerate}
	\item $K(y,v)\geq 0$ for all $v \in \partial X$ and $y\in X$,
	\item $\int_{\partial X} K(y,v) d\mu_{x}(v) =1$ for all $y \in X$,
	\item for all $ v_{0} \in \partial X$ and for all $ r_{0} >0$ we have: $$\int_{\partial X \backslash B(v_{0},r_{0})}K(y,v)d\mu_{x}(v) \rightarrow 0 ~~\mbox{as}~~ y \rightarrow v_{0}. $$
\end{enumerate}
\end{defi} 

A Dirac-Weierstrass family yields an integral operator $\mathcal{K}$:
\begin{align*}
 \mathcal{K}:   f \in L^{1}(\partial X,\mu_{x}) & \mapsto \mathcal{K}f \in C(X)
\end{align*}
defined as :
\begin{align*}
\mathcal{K}f :  y \in  X &\mapsto  \int_{\partial X}f(v)K(y,v)d\mu_{x}(v) \in \mathbb{C} .
\end{align*}
\subsection{Continuity}\label{sectioncontinuity}

Let $f$ be a function on $\partial X$. We define the function $\overline{\mathcal{K}}f$ on $\overline{X}$ as the following: 
\begin{equation}\label{extension}
\overline{\mathcal{K} }f: y \in  \overline{X} \mapsto  \overline{\mathcal{K}}f(y)= \left\{
    \begin{array}{ll}
       \mathcal{K}f(y) & \mbox{if } y \in X \\
        f(y) & \mbox{if }  y \in \partial X 
    \end{array}
		\right. 
		\end{equation}

Thus, $\overline{\mathcal{K}}$ is an operator which assigns a function defined on $\overline{X}$ to a function defined on $\partial X$.

\begin{prop}\label{continuity}
If $f$ is a continuous functions on $\partial X$, the function $\overline{\mathcal{K}}  (f)$ is a continuous function on $\overline{X}$.
 
\end{prop}
\begin{proof}
 Observe first that since $K$ is a continuous function on $X$ the function  $\overline{\mathcal{K}}f$ is on $X$. \\
 Let $v_{0}$ be in $\partial X$ and let $\epsilon>0$.  
Since $f$ is continuous, there exists $r>0$ such that $$|f(v_{0})-f(v)|<\frac{\epsilon}{2},$$ whenever $v\in B(v_{0},r)$. Besides, by (3) in Definition \ref{diracweier}, there exists a neighborhood $V$ of $v_{o}$ such that for all $y\in V$ we have: $$\int_{\partial X \backslash B(v_{0},r) } K(y,v) d\mu(v)\leq \frac{\epsilon}{4\|f\|_{\infty}}\cdot$$ We have for all $y\in V$ : 
\begin{align*}
|\overline{\mathcal{K}}f(v_{0})-\overline{\mathcal{K}}f(y)| &\leq \int_{ B(v_{0},r) }|f(v_{0})-f(v)| K(y,v) d\mu_{x}(v)+\int_{\partial X \backslash B(v_{0},r) }|f(v_{0})-f(v)| K(y,v) d\mu_{x}(v) \\ 
&\leq \frac{\epsilon}{2} +2\|f\|_{\infty}\int_{\partial X \backslash B(v_{0},r) } K(y,v) d\mu_{x}(v)\\
&\leq \epsilon.
\end{align*}

Hence, $\overline{\mathcal{K}}f$ is a continuous function on $\overline{X}$. 
\end{proof}

\subsection{Examples of Dirac-Weierstrass family}
Let $R>0$, and consider for each $y \in X$ a point $w_{x}^{y}\in \mathcal{O}_{R}(x,y)$.
We start by a lemma:
\begin{lemma}\label{visual} Let $v_{0}$ be in $\partial X$. Then $d_{x}(v_{0},w_{x}^{y})\rightarrow 0$ as $y\rightarrow v_{0}$.
\end{lemma}
\begin{proof}
 Let $y_{n}$ be a sequence of points of $X$ such that $y_{n}\rightarrow v_{0}$. 
 Apply the right hand side inequality of Lemma \ref{encadrement buseman dur} to get
$$
 (v_{0},w^{y_{n}}_{x})_{x}\geq (v_{0},y_{n})_{x}-R-\delta.
 $$
 Since $y_{n}\rightarrow v_{0}$, we have $(v_{0},y_{n})_{x}$ goes to infinity, and thus $d_{x}(v_{0},w_{x}^{y})\rightarrow 0$ as $y\rightarrow v_{0}$.
\end{proof}

\begin{prop}\label{dirac}
Assume that there exists a polynomial $Q_{1}$ (at least of degree 1)  such that for all $y\in X$ with $d(x,y)$ large enough $Q_{1}\big(d(x,y)\big)>0$ and $$Q_{1}\big(d(x,y)\big)\exp{\bigg(-\frac{\alpha}{2}d(x,y)\bigg)}\leq P_{0}\textbf{1}_{\partial X}(y).$$ Then  $$\bigg( \frac{P(y,.)^{1/2}}{P_{0}\textbf{1}_{\partial X}(y)} \bigg)_{y\in X} $$ is a Dirac-Weierstrass family.
\end{prop}
\begin{proof}
Let $B(v_{0},r_{0})$ the ball of radius $r_{0}$ at $v_{0}$ in $\partial X$ with respect to  $d_{x}$.\\
 Let $\epsilon>0$. Since $Q_{1}$ is a polynomial at least of degree one, there exists $R'>0$ such that for all $y$ satisfying $d(x,y)>R'$ we have $$\frac{ C_{r_{0},\alpha,\delta}\|\mu_{x} \|}{Q_{1}\big(d(x,y)\big)}<\epsilon$$ where  $C_{r_{0},\alpha,\delta}=2^{\alpha}\exp{(\alpha (\delta+R))}/r_{0}^{\alpha}$ is a positive constant.

 Lemma \ref{visual} yields a neighborhood $V$ of $v_{0}$ such that $d_{x}(v_{0},w_{x}^{y})\leq r_{0}/2$ for all $y\in V$. We have  for all $v$ in $\partial X \backslash B(v_{0},r_{0})$:
\begin{align*}
d_{x}(v,w_{x}^{y})&\geq d_{x}(v,v_{0})- d_{x}(v_{0},w_{x}^{y})\\
&\geq r_{0}-d_{x}(v_{0},w_{x}^{y})\\
&\geq \frac{r_{0}}{2}.
\end{align*}
We set $V_{R'}=V\cap X\backslash B_{X}(x,R')$.
Combining Lemma \ref{ineghch} with  the above inequality we obtain for all $y\in V_{R'}$:
\begin{align*}
\int_{\partial X \backslash B(v_{0},r_{0})}\frac{P^{\frac{1}{2}}(y,v)}{P_{0}\textbf{1}_{\partial X}(y)}d\mu_{x}(v) &\leq 
\int_{\partial X \backslash B(v_{0},r_{0})}\frac{\exp{\big(\alpha (\delta+R)\big)}\exp{\big(-\frac{\alpha}{2} d(x,y)\big)}}{d_{x}^{\alpha}(v,w_{x}^{y})\big(P_{0}\textbf{1}_{\partial X}(y) \big)} d\mu_{x}(v)\\
&\leq C_{r_{0},\alpha,\delta}  \int_{\partial X \backslash B(v_{0},r_{0})}\frac{\exp{\big(-\frac{\alpha}{2} d(x,y)\big)}}{Q_{1}\big(d(x,y)\big)\exp{\big(-\frac{\alpha}{2} d(x,y)\big)}}d\mu_{x}(v)\\
&= C_{r_{0},\alpha,\delta}\int_{\partial X \backslash B(v_{0},r_{0})}\frac{1}{Q_{1}\big(d(x,y)\big)}d\mu_{x}(v)\\
&\leq\frac{ C_{r_{0},\alpha,\delta}\mu_{x}(\partial X)}{Q_{1}\big(d(x,y)\big)}\\
&\leq \epsilon.
\end{align*}

It follows that $$\int_{\partial X \backslash B(v_{0},r_{0})}\frac{P^{\frac{1}{2}}(y,v)}{P_{0}\textbf{1}_{\partial X}(y)}d\mu_{x}(v)\rightarrow 0 \mbox{ as } y\rightarrow v_{0}.$$

\end{proof}

Besides, the same method proves the following proposition: 
\begin{prop}\label{poissondirac}
Assume that there exists a constant $C>0$ such that $$\frac{\|\mu_{x} \|}{\| \mu_{y} \|} \leq C,$$ for all $y\in X$. 
The normalized Poisson kernel $$\bigg(\frac{P(y,.)}{\| \mu_{y} \| } \bigg)_{y\in X}$$ is a Dirac-Weierstrass family.
\end{prop}

\subsection{Equidistribution theorems extended to $(L^{1})^{*}$}

Theorem \ref{roblin} of T. Roblin  has for immediate consequence:
\begin{theorem}\label{conseqro1}(T. Roblin)
Let $\Gamma$ be a discrete subgroup of isometries of $X$ with a non-arithmetic spectrum. Assume that $\Gamma$ admits a finite BMS measure associated to a $\Gamma$-invariant conformal density $\mu$ of dimension $\alpha=\alpha(\Gamma)$. Then for each $x \in X$ and for all $\rho>0$ we have as $n$ goes to infinity: $$\frac{1}{|C_{n}(x,\rho)|}\sum_{ C_{n}(x,\rho)}D_{\gamma^{-1}  x}  \rightharpoonup \frac{\mu_{x}}{\|\mu_{x} \|} $$ with respect to  the weak* topology of $C(\overline X)^*$.
\end{theorem}

We view Theorem D as new new equidistribution theorem, where the weak* convergence is not on the dual of the space of continuous functions but rather on the dual of space of $L^{1}$ functions on the boundary.

\begin{proof}[Proof of Theorem D]
Let $x$ in $X$ and $\rho>0$, and consider $N_{x,\rho}$ such that $n\geq N_{x,\rho}$ implies $|C_{n}(x,\rho)|>0$.
We give a proof for the densities $(\mu_{x})_{x \in X}$. For all $n\geq N_{x,\rho}$, we denote by $\lambda_{x,\rho}^{n}$ the following measure $$\lambda_{x,\rho}^{n}=\frac{1}{|C_{n} (x,\rho)|}\sum_{\gamma \in C_{n}(x,\rho)}\frac{\mu_{\gamma  x}}{\| \mu_{\gamma x} \|} .$$\\

\textbf{Step 1: the sequence of measures $(\lambda_{x,\rho}^{n})_{n\geq N_{x,\rho}}$ is uniformly bounded.}

Since the dual space of $L^{1}(\partial X,\mu_{x})$ is $L^{\infty}(\partial X,\mu)$ we have for $n\geq N_{x\rho}$:
\begin{align*}
 \|H^{n}_{x,\rho} \|_{\infty}&=\sup_{\|f\|_{1}\leq 1} \bigg \lbrace \bigg | \int_{\partial X}H^{n}_{x,\rho}(v)f(v)d\mu_{x}(v) \bigg | \bigg \rbrace \\
 &= \sup_{\|f\|_{1}\leq 1} \bigg \lbrace \bigg | \frac{1}{|C_{n} (x,\rho)|}\sum_{\gamma \in C_{n}(x,\rho)}\frac{\mu_{\gamma  x}}{\| \mu_{x} \|}(f) \bigg | \bigg \rbrace \\
 &= \| \lambda_{x,\rho}^{n}\|_{(L^{1})^{*}}. 
  \end{align*} 
 
Hence Proposition \ref{inegsample2} completes \textbf{Step 1}.

\textbf{Step 2:  computation of the limit of $(\lambda_{n}^{x})_{n\geq N_{x,\rho}}$.}  \\
By Banach-Alaoglu's theorem, $(\lambda_{x,\rho}^{n})_{n \geq N_{x,\rho}}$ has accumulation points. Denote by $\lambda_{x}^{\infty}$ such accumulation point. Let $f\in C(\partial X)$, Proposition \ref{poissondirac} combined with Proposition \ref{continuity} define $\overline{\mathcal{P}}f$ as a continuous function on $\overline{X}$ (as in (\ref{extension}) in Subsection \ref{continuity}), where $\mathcal{P}$ is associated with the normalized Poisson kernel defined as in Proposition \ref{poissondirac}. We have for all $n\geq N_{x,\rho}$:
\begin{align*}
\lambda_{x,\rho}^{n}(f)&=\frac{1}{|C_{n} (x,\rho)|}\sum_{\gamma \in C_{n}(x,\rho)} \frac{\mu_{\gamma  x}(f)}{\| \mu_{\gamma x} \|} \\
&=\frac{1}{|C_{n} (x,\rho)|}\sum_{\gamma \in C_{n}(x,\rho)}\mathcal{P}(f)(\gamma x)\\
&= \frac{1}{|C_{n} (x,\rho)|}\sum_{\gamma \in C_{n}(x,\rho)}D_{\gamma  x}\big(\overline{\mathcal{P}}(f)\big).
\end{align*} 
Applying Roblin's theorem \ref{conseqro1} by taking the limit in the above inequality, we obtain for all $f\in C(\partial X)$

$$\lambda_{x}^{\infty}(f)=\mu_{x}\big(\overline{\mathcal{P}}(f)\big)=\frac{\mu_{x}(f)}{\|\mu_{x} \|}.$$
Since $C(\partial X)$ is dense $L^{1}(\partial X,\mu_{x})$ with respect to  the $L^{1}$ norm, we deduce that $(\lambda_{x,\rho}^{n})_{n\geq N_{x,\rho}}$ has only one accumulation point which is $\mu_{x}$, and the proof is done.

The proof concerning $(\nu_{x})_{x\in X}$ follows the same method, and uses $\varphi_{x}=P_{0}$ in order to have available Proposition \ref{dirac} for $\Gamma$ in $\mathcal{C}$. Indeed, since the lower bound of the Harish-Chandra estimates holds a priori only on $CH(\Lambda_{\Gamma}) \backslash B_{X}(x,R_{x})$ we rather use Proposition \ref{dirac} with $\overline{CH(\Lambda_{\Gamma})}=CH(\Lambda_{\Gamma})\cup \Lambda_{\Gamma}$ instead of $\overline{X}=X \cup \partial X$.  If $f$ is a continuous function on $\Lambda_{\Gamma}$, the function $\overline{\mathcal{P}_{0}}f$ on $\overline{CH(\Lambda_{\Gamma})}$ defined as 
\begin{equation}\label{extension}
\overline{\mathcal{P}_{0} }f: y \in  \overline{CH(\Lambda_{\Gamma})} \mapsto  \overline{\mathcal{P}_{0}}f(y)= \left\{
    \begin{array}{ll}
       \mathcal{P}_{0}f(y) & \mbox{if } y \in CH(\Lambda_{\Gamma}) \\
        f(y) & \mbox{if }  y \in  \Lambda_{\Gamma}
    \end{array}
		\right. 
		\end{equation} 
is continuous on $\overline{CH(\Lambda_{\Gamma})}$.
\end{proof}

\begin{remark}
We may ask if an analogous theorem of Theorem \ref{roblin} dealing with $\mu_{x}$ instead of the Dirac mass holds  (assuming $\| \mu_{x} \|=1$ for simplicity)? More precisely do we have: $$\frac{1}{|C_{n} (x,\rho)|}\sum_{\gamma \in C_{n}(x,\rho)}\mu_{\gamma ^{-1} x} \otimes \mu_{\gamma  x} \rightharpoonup \mu_{x} \otimes \mu_{x}$$ with respect to  the weak* convergence of $L^{1}(\partial X\times \partial X,\mu_x \otimes \mu_{x})^{*}$ (for some $\rho$)? The answer is negative because a duality argument combined with Banach-Steinhaus theorem would imply that the sequence of functions $$G_{n}:(v,w)\mapsto \frac{1}{|C_{n} (x,\rho)|}\sum_{\gamma \in C_{n}(x,\rho)} \exp(\alpha \beta_{v}(x,\gamma^{-1} x))\exp(\alpha \beta_{w}(x,\gamma x))$$ is uniformly  bounded with respect to  the $L^{\infty}(\mu)$ norm. It is easy to see that this is impossible by evaluating $G_{n}$ at $(v,w)\in \mathcal{O}_{R}(x,\gamma^{-1} x) \times \mathcal{O}_{R}(x,\gamma x)$ for some $\gamma \in C_{n}(x,\rho)$. We obtain the same answer to the same question dealing with $\nu_{x}$ by considering the sequence of functions $$(v,w)\mapsto \frac{1}{|C_{n} (x,\rho)|}\sum_{\gamma \in C_{n}(x,\rho)} \frac{\exp(\frac{\alpha}{2} \beta_{v}(x,\gamma^{-1}  x))\exp(\frac{\alpha}{2} \beta_{w}(x,\gamma x))}{\phi_{x}^{2}(\gamma)} \cdot $$ 
\end{remark}

\end{document}